\documentclass[a4paper,aip,cha,reprint,author-year,floatfix]{revtex4-1}
\pdfoutput=1
\usepackage{amsmath,amssymb,amsthm}
\usepackage{bm}
\usepackage{float}
\usepackage{hyperref}
\hypersetup{
  colorlinks=true,
  citecolor=blue,
  linkcolor=black,
  urlcolor=blue,
  pdftitle={Isochronous Dynamics in Pulse Coupled Oscillator Networks with Delay},
  pdfauthor={Pan Li, Wei Lin, Konstantinos Efstathiou},
}

\usepackage{tabularx}
\usepackage{array}

\usepackage{graphicx}
\graphicspath{{figures/}{notebook/}}
\usepackage{epstopdf}
\usepackage{xcolor}

\usepackage[caption=false]{subfig}

\usepackage{enumitem}

\newtheorem{lemma}{Lemma}[section]
\newtheorem{prop}[lemma]{Proposition}

\theoremstyle{definition}
\newtheorem{defn}[lemma]{Definition}
\newtheorem{remark}[lemma]{Remark}

\newcommand{\wOmega}{\widetilde\Omega}

\begin{document}

\title{Isochronous Dynamics in Pulse Coupled Oscillator Networks with Delay}

\author{Pan~Li}
\email{P.Li@rug.nl}
\affiliation{School of Mathematical Sciences, Centre for Computational Systems Biology of ISTBI, Fudan University, Shanghai 200433, China}
\affiliation{Johann Bernoulli Institute for Mathematics and Computer Science,
University of Groningen, P.O. Box 407, 9700 AK, Groningen, The Netherlands}

\author{Wei~Lin}
\email{wlin@fudan.edu.cn}
\affiliation{School of Mathematical Sciences, Centre for Computational Systems Biology of ISTBI, Fudan University, Shanghai 200433, China}
\affiliation{Shanghai Key Laboratory of Contemporary Applied Mathematics, and LMNS, Ministry of Education, China}

\author{Konstantinos~Efstathiou}
\email{K.Efstathiou@rug.nl}
\affiliation{Johann Bernoulli Institute for Mathematics and Computer Science,
University of Groningen, P.O. Box 407, 9700 AK, Groningen, The Netherlands}

\date{\today}

\begin{abstract}
  We consider a network of identical pulse-coupled oscillators with delay and all-to-all coupling. We demonstrate that the discontinuous nature of the dynamics induces the appearance of \emph{isochronous regions}---subsets of the phase space filled with periodic orbits having the same period. For fixed values of the network parameters each such isochronous region corresponds to a subset of initial states on an appropriate surface of section with non-zero dimension such that all periodic orbits in this set have qualitatively similar dynamical behaviour. We analytically and numerically study in detail such an isochronous region, give a proof of its existence, and describe its properties. We further describe other isochronous regions that appear in the system.
\end{abstract}


\maketitle

\begin{quotation}
  Pulse coupled oscillator networks are a key model for the study of synchronization in a wide variety of systems, ranging from fireflies to wireless communication systems. Moreover, despite their simplicity, they manifest dynamical behavior that does not typically appear in smooth finite-dimensional dynamical systems. We report on the existence of isochronous dynamics in pulse coupled oscillator networks with delay: for suitable values of the parameters there exist open sets of initial conditions giving periodic orbits with the same period. This, previously unknown, behavior of pulse coupled oscillator networks with delay provides a deeper understanding of their dynamics and how they can reach synchronization.  
\end{quotation}

\section{Introduction}

\paragraph{Pulse coupled oscillator networks} Pulse coupled oscillator networks (PCONs) have been used to model interactions in networks where each node affects other nodes in a discontinuous way. Two such examples are the synchronization related to the function of the heart \citep{Peskin1975} and the synchronization of fireflies \citep{Mirollo1990}. There is now an extensive literature on the dynamics of pulse coupled oscillator networks focusing on synchronization and the stability of synchronized states.

Concerning syncronization, after the seminal work \citet{Mirollo1990} who considered excitatory coupling with no delay, \citet{Ernst1995, Ernst1998} showed the importance of delayed and inhibitory coupling for complete synchronization, while excitatory coupling leads to synchronization with a phase lag. In particular, for inhibitory coupling it was shown that the network syncronizes in multistable clusters of common phase. \citet{Wu2007, Wu2009} showed that all-to-all networks with delayed excitatory coupling do not synchronize, either completely or in a weak sense, for sufficiently small delay and coupling strength. In \citep{Wu2010a} it was shown that the parameter space in systems with excitatory coupling is separated into two regions that support different types of dynamics. The effect of network connectivity to synchronization is numerically studied in \citep{LaMar2010} where it is shown that the proportion of initial conditions that lead to synchronization is an increasing function of the node-degree. \citet{Kielblock2011} showed that pulses induce the breakdown of order preservation, and demonstrated a system of 2 identical and symmetrically coupled oscillators where the winding numbers of the two oscillators can be different. \citet{Klinglmayr2012a} showed that under self-adjustment assumptions, systems with heterogeneous phases rates and random individual delays would converge to a close-to-synchrony state. Moreover, synchronization has been considered in systems with stochastic features. \citet{OKeeffe2015} studied how small clusters of synchronized oscillators in all-to-all networks coalesce to form larger clusters and obtained exact results for the time-dependent distribution of cluster sizes.

Except for synchronized states more interesting dynamics also manifests in pulse coupled oscillator networks. The existence of unstable attractors has been established, numerically and analytically, in all-to-all pulse coupled oscillator networks with delay, see \citep{Ashwin2005, Broer2008a, Timme2002a, Timme2003}. Unstable attractors are fixed points or periodic orbits, which are locally unstable, but have a basin of attraction which is an open subset of the state space. Heteroclinic connections between saddle periodic orbits, such as unstable attractors, have been shown to exist in pulse coupled oscillator networks with delay \citep{Ashwin2004, Ashwin2005b, Broer2008} and they have been proposed as representations of solutions of computational tasks. \citet{Schittler-Neves2012} showed that complex networks of dynamically connected saddle states are capable of computing arbitrary logic operations by entering into switching sequences in a controlled way. \citet{Timme2008} gave an analysis of asymptotic stability for topologically strongly connected PCONs, while \citet{Zeitler2009} analyzed the influence of asymmetric coupling and showed that it leads to a smaller bistability range of synchronized states. \citet{Zumdieck2004} numerically showed the existence of long chaotic transients in pulse-coupled oscillator networks. The length of the transients depends on the network connectivity and such transients become prevalent for large networks.

\paragraph{Isochronous dynamics}

\begin{figure}[htbp]
  \centering
  \includegraphics[width=0.8\linewidth]{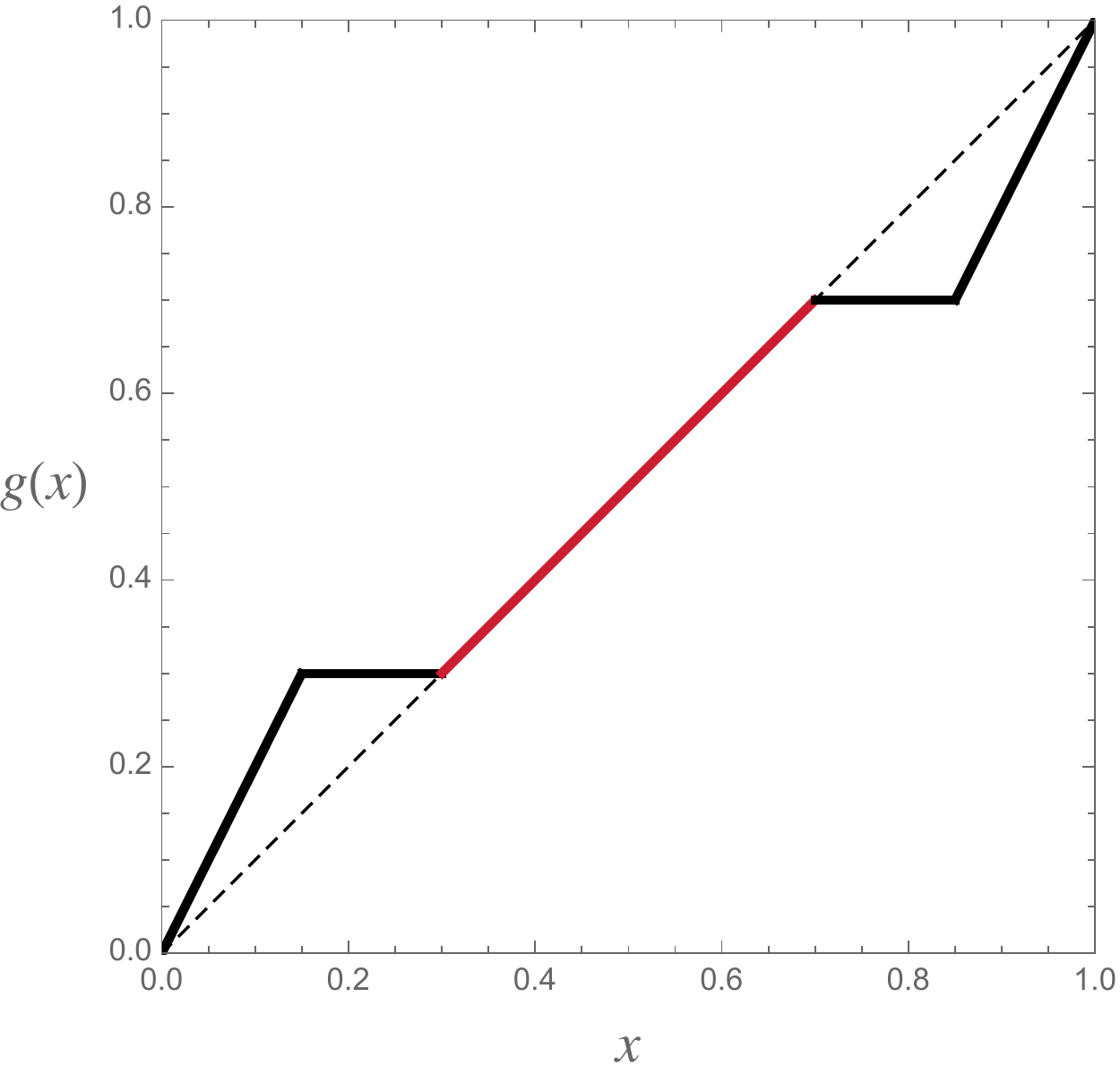}
  \caption{Isochronous dynamics in a non-smooth 1D map.}
  \label{fig/1d-isochronous}
\end{figure}

In this paper we report on a newly observed dynamical behavior of PCONs with delay. Specifically, we show that for appropriate values of the coupling parameters, that is, of the coupling strength $\varepsilon$ and the delay $\tau$, there is a $n\ge1$-dimensional subset of state space foliated by periodic orbits having the same period. We call the subsets of state space \emph{isochronous regions}. These periodic orbits are equivalent in a sense we make precise in Definition~\ref{defn/pulse-eqv}. Furthermore, the parameter region for which such periodic orbits manifest is an open subset of the parameter space.

This type of observed dynamics in PCON with delay is a special case of \emph{isochronous dynamics}. One talks of isochronous dynamics when a dynamical system has an open set of initial states that give rise to periodic solutions having the same period. Examples include the one-dimensional harmonic oscillator, any $N$-dimensional harmonic oscillator where the frequencies, $\omega_1,\dots,\omega_N$, satisfy $N-1$ resonance relations, and the restriction of the Kepler problem to any constant energy surface. We refer to \citep{Calogero2011} for an extensive review of recent results pertaining to isochronous dynamics in the context of ordinary differential equations and Hamiltonian systems. Nevertheless, such isochronous dynamics have not been previously observed in PCONs, except of course for the trivial case of identical uncoupled oscillators.

A non-trivial example of isochronous dynamics induced by a non-smooth map $g: [0,1] \to [0,1]$ is depicted in \autoref{fig/1d-isochronous}. Each point in the middle (red) segment of the graph of $g$, lying along the diagonal, is a fixed point of $g$ and thus such points give isochronous dynamics of period $1$. 

\paragraph{Structure of the paper} In \autoref{sec/pcon} we describe the dynamics of PCONs with delay and we review its basic properties. In \autoref{sec/per-plat} we first present numerical experiments that show the appearance of $n\ge1$-dimensional sets of periodic orbits on a surface of section for specific values of the dynamical parameters. Then we define the notion of a \emph{isochronous region}. In \autoref{sec/ir4} we discuss in detail one of the isochronous regions in the system. We prove its existence for an open subset of parameter values, describe in detail the dynamics in the region, and determine the stability of the periodic orbits that constitute the region. In \autoref{sec/other} we briefly describe other isochronous regions that appear in the system. We conclude the paper in \autoref{sec/conclusions}.

\section{Dynamics of PCONs with delay}
\label{sec/pcon}

In this section we specify the dynamics of the PCONs with delay that we consider in this paper. 

\subsection{Mirollo-Strogatz model with delay}

We consider a variation of the Mirollo-Strogatz model \citep{Mirollo1990, Ernst1995, Ernst1998}. The system here is a homogeneous all-to-all network consisting of $N$ pulse coupled oscillators with delayed excitatory interaction. All the oscillators follow the same integrate-and-fire dynamics. Between receiving pulses the state of each oscillator evolves autonomously and its dynamics is smooth. When the $i$-th oscillator reaches the threshold value $x_i = 1$ its state is reset to $x_i=0$. At the same moment the $i$-th oscillator sends a pulse to all other oscillators, $j \ne i$, in the network. The time between the moment an oscillator sends a pulse and the moment the other oscillators receive that pulse is the \emph{delay} $\tau \ge 0$. When the $i$-th oscillator receives $m$ simultaneous pulses without crossing the threshold value, its state variable jumps to $x_i' = x_i + m \hat{\varepsilon}$. If $x_i + m \hat{\varepsilon} \ge 1$, that is, if the oscillator crosses the firing threshold by receiving these pulses, then the new state becomes $x_i' = 0 \equiv 1$. The dynamics for each oscillator is thus given by
\begin{subequations}
  \begin{align}
    \dot{x}_i(t) & = F(x_i(t)), \\
    x_i(t^+) & = 0,\; \text{if $x_i(t)=1$}, \\
    \intertext{and}
    x_i (t) & = \min(1, x_i (t^-) + m\hat{\varepsilon} ),
  \end{align}
  if $m$ other oscillators fired at time $t-\tau$.
\end{subequations}
Here, $\hat{\varepsilon} = \varepsilon / (N-1)$, where $\varepsilon \ge 0$ is the \emph{coupling strength}, and $F$ is a positive, decreasing, function $(F > 0,\, F' < 0)$.

To simplify the description of the dynamics we define, following \cite{Mirollo1990}, the phases $(\theta_i)_{i=1}^N$ instead of the state variables $(x_i)_{i=1}^N$.
The two sets of variables are related through
\begin{align*}
  x_i = f(\theta_i),
\end{align*}
where $f:[0,1]\to[0,1]$ is a diffeomorphism fixing the endpoints, that is, $f(0)=0$ and $f(1)=1$. The map $f$ is defined through the requirement that the uncoupled dynamics of each oscillator is given by $\dot{\theta}_i = 1$. This implies
\begin{align*}
  \dot x = F(x) = f'(f^{-1}(x)) = \frac{1}{(f^{-1})'(x)},
\end{align*}
and that $f$ is increasing and concave down $(f' > 0,\, f'' < 0)$.
Following \cite{Mirollo1990} we choose
\begin{align*}
  F(x) := F_b(x) = \frac{e^b - 1}{b} e^{-bx}, \quad b > 0,
\end{align*}
giving
\begin{align*}
  f(\theta) := f_b(\theta) = \frac{1}{b} \ln \left( 1 + (e^b-1)\, \theta \right).
\end{align*}

Then, in terms of the phases $\theta_i$, the dynamics is given by
\begin{subequations}
  \begin{align}
    \dot{\theta}_i(t) & = 1, \\
    \theta_i(t^+) & = 0,\; \text{if $\theta_i(t)=1$,} \\
    \intertext{and}
    \theta_i(t) & = \min\{ 1, H(\theta_i(t^-),m\hat{\varepsilon}) \},\;
  \end{align}
  if $m$ other oscillators fired at time $t-\tau$.
\end{subequations}
The function $H$ is defined by
\begin{align}\label{eq/H-func}
  H (\theta, \delta)
  = f^{-1}(f(\theta)+\delta)
  = e^{b \delta} \, \theta + \frac{e^{b \delta}-1}{e^b-1},
\end{align}
and it gives the new phase of an oscillator with phase $\theta$ after it receives a pulse of size $\delta$, ignoring the effect of the threshold.

\begin{figure}[tbp]
  \subfloat[\label{fig/f-plot}]%
  {\includegraphics[width=0.47\linewidth]{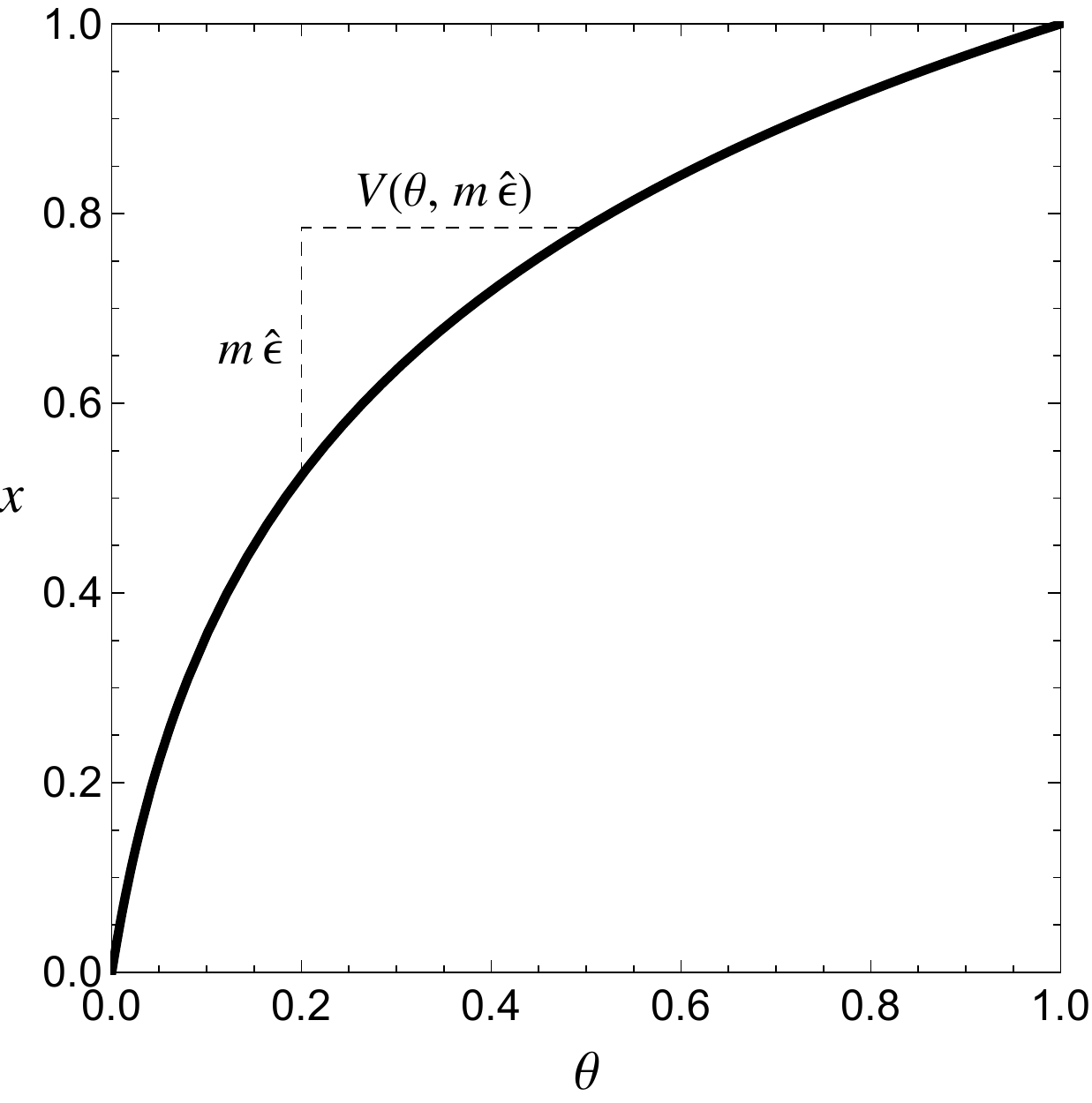}}
  \hfill
  \subfloat[\label{fig/V-plot}]%
  {\includegraphics[width=0.48\linewidth]{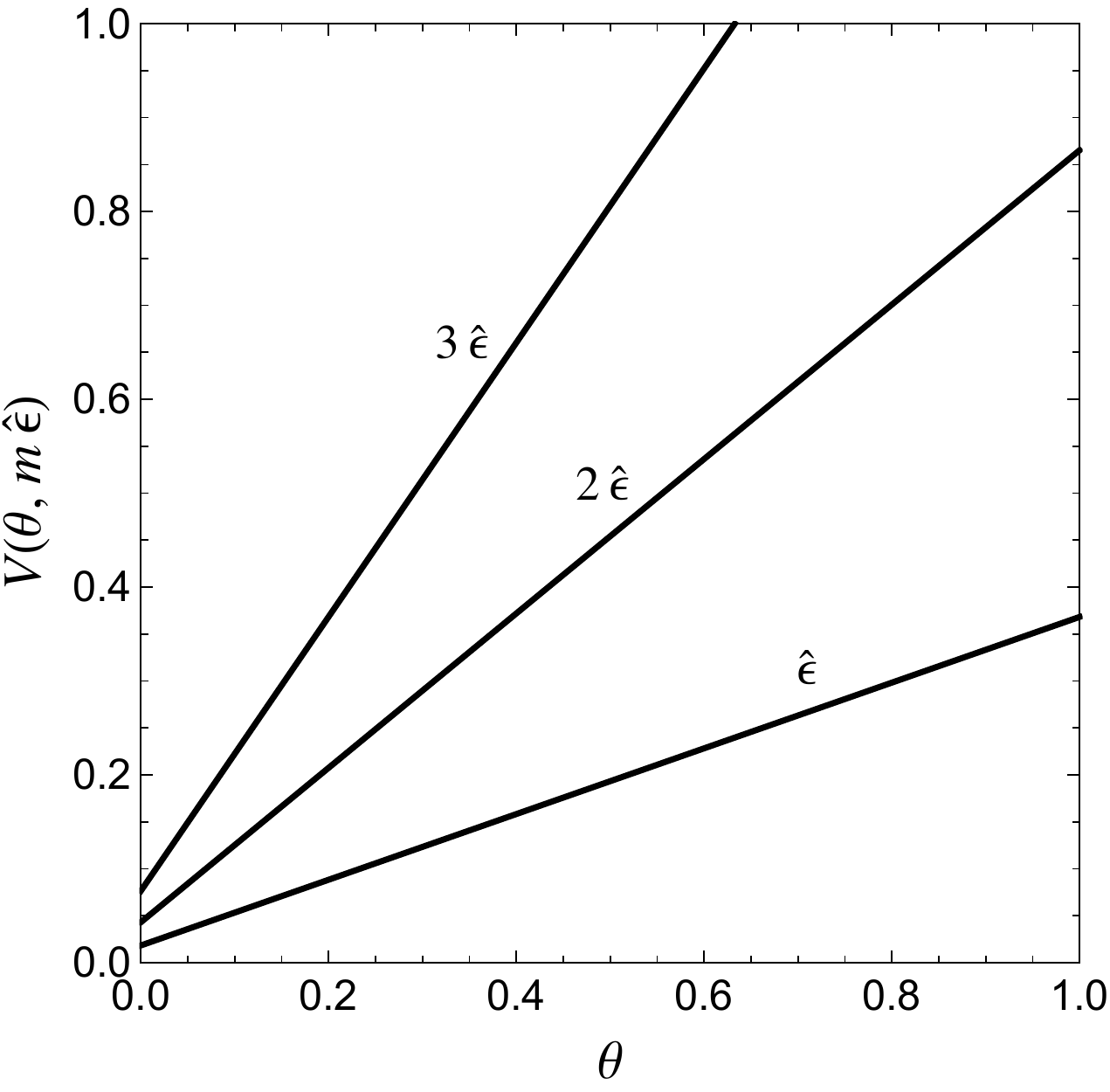}}
  \caption{(a) The function $f_b(\theta)$ for $b=3$. (b) The function $V(\theta,m\hat{\varepsilon})$ for $b=3$, $\hat{\varepsilon}=0.1$, and $m=1,2,3$.\label{PRC}}
\end{figure} 

Typically, one also defines the \emph{pulse response function} (PRF) $V(\theta, \delta)$ representing the change in phase after receiving a pulse of size $\delta$, ignoring the effect of the threshold. Specifically,
\begin{align}\label{eq/prf}
  V(\theta,\delta) = H(\theta,\delta) - \theta = (e^{b\delta} -1)\, \theta + \frac{e^{b\delta}-1}{e^b-1},
\end{align}
see \autoref{fig/V-plot}.
Note that the function $H$ in Eq.~\eqref{eq/H-func} has the property
\begin{align*}
  H(H( \theta,\, \delta),\, \delta') = H(\theta, \delta+\delta').
\end{align*}
implying
\begin{align*}
  H(H( \theta,\, m\,\hat{\varepsilon}),\, m'\,\hat{\varepsilon}) = H(\theta, (m+m')\,\hat{\varepsilon}).
\end{align*}
To simplify notation, for fixed value of $\hat{\varepsilon}$, we write
\begin{align*}
  H(\theta,m\hat{\varepsilon}) = H_m(\theta) \ \ \text{and}\ \ H(\theta,\hat{\varepsilon}) = H_1(\theta) = H(\theta).
\end{align*}

\subsection{Description of the dynamics}

In principle, to determine the dynamics of a system with delay $\tau$ for $t \ge 0$ one should know the phases $\theta_i(t)$, $i=1,\dots,N$ of the oscillators for all $t \in [-\tau,0]$. This information can be encoded in the \emph{phase history function}
\begin{align*}
  \theta: [-\tau,0] \to \mathbb T^n : t \mapsto (\theta_1(t), \dots, \theta_N(t)).
\end{align*}

In the particular system studied here, this description can be further simplified since it is not all the information about the phases in $[-\tau,0]$ that is necessary to determine the future dynamics. Instead, it is enough to know the phases $\theta_i(0)$, $i=1,\dots,N$ at $t=0$ and the \emph{firing moments} of each oscillator in $[-\tau,0]$, that is, the moments when each oscillator reaches the threshold value.

We denote by $-\sigma_i^{(j)}$ the $j$-th firing moment of the $i$-th oscillator in $[-\tau,0]$ and by $\Sigma_i = \{ \sigma_i^{(j)} \}$ the set of all such firings moments. Note that our ordering is
\begin{align*}
  \cdots < - \sigma_i^{(3)} < - \sigma_i^{(2)} < - \sigma_i^{(1)} \le 0.
\end{align*}
To simplify notation we also write $\sigma_i = \sigma_i^{(1)}$ in the case that $\Sigma_i$ contains exactly one element. We call the $\sigma_i^{(j)}$ \emph{firing time distances} (FTD) and $\sigma_i$ the \emph{last firing time distance} (LFTD).

\begin{remark}
  It is shown in \cite{Ashwin2005} that for sufficiently small values of $\varepsilon$ and $\tau$ the size of the set $\Sigma = \bigcup_{i=1}^N \Sigma_i$ is bounded for all $t \ge 0$. The parameter region of interest in the present paper is not covered by the explicit estimates given in \citep{Ashwin2005}. Nevertheless, for the specific orbits in the isoschronous regions we consider, the size of $\Sigma$ remains bounded for all $t \ge 0$.
\end{remark}

The dynamics of the system for $t \ge 0$ can then be determined from the FTD in $[-\tau,0]$ and the phases at $t=0$, i.e., from the set
\begin{align*}
  \Phi = \bigl\{ \{ \sigma_i^{(j)} \}_j, \theta_i \bigr\}_{i=1,\dots,N}.
\end{align*}
When $\Phi$ is a finite set we can ask whether a neighborhood is a finite or infinite dimensional set. \citet{Broer2008a} show that, choosing an appropriate metric on the space of phase history functions, a neighborhood of $\Phi$ is finite dimensional. Nevertheless, this local dimension is not constant and is not bounded throughout the state space.

To describe high-dimensional dynamics it is convenient to introduce a \emph{Poincar\'e surface of section}. Here we choose the surface $\theta_N = 0$, see also \citep{Ashwin2005, Broer2008a}. Given a state $\Phi$ with $\theta_N = 0$ the time evolution of the system produces a new state $\Phi'$ when $\theta_N$ becomes again $0$. This defines the \emph{Poincar\'e map} $\mu: \Phi \to \Phi'$. We call the sequence of points $\mu^j(\Phi)=\mu\left( \mu^{j-1}(\Phi) \right)$, $j=1,2,\dots$, the \emph{Poincar\'e orbit} with initial state $ \mu^{0}(\Phi)=\Phi$. We also define a related concept.

\begin{defn}[Phase orbit]
  Consider a Poincar\'e orbit $\{ \mu^j(\Phi) \}_{j=0,1,2,\dots}$ and let $\mathrm{pr}_\theta$ denote the projection
  \begin{align*}
    \Phi = \bigl\{ \{ \sigma_i^{(j)} \}_j, \theta_i \bigr\}_{i=1,\dots,N} \mapsto \{ \theta_i \}_{i=1,\dots,N-1}. 
  \end{align*}
  Then the \emph{phase orbit} of $\Phi$ is the sequence $\mathrm{pr}_\theta(\mu^j(\Phi))$, $j=0,1,2,\dots$.
\end{defn}

\begin{remark}
  Note that the phase orbit gives only a projection of the dynamics to the space of $N-1$ phase variables $(\theta_1,\dots,\theta_{N-1})$. Since the full dynamics further depends on the firing moments in the time interval $[-\tau,0]$ we cannot define a map $\mathbb T^{N-1} \to \mathbb T^{N-1}$ that depends only on the phases $(\theta_1,\dots,\theta_{N-1})$ and fully encodes the dynamics.
\end{remark}

A Poincar\'e orbit $\{ \mu^j(\Phi) \}_{j=0,1,2,\dots}$ for which $\mu^{j+T_P}(\Phi) = \mu_j(\Phi)$ for all $j \ge 0$ is called \emph{periodic} with Poincar\'e period $T_P$. Note that $T_P$ is not necessarily the minimal period.  By construction, a periodic Poincar\'e orbit corresponds to a periodic orbit in the full state space for the dynamics with continuous time $t \ge 0$. In particular, let $\Phi(t)$ be the state at time $t \ge 0$ corresponding to a periodic Poincar\'e orbit. Then there is a time $T$, corresponding to $T_P$, such that $\Phi(t+T) = \Phi(t)$ for all $t \ge 0$. We call $T$ the \emph{orbit period}.

\section{Isochronous Dynamics}
\label{sec/per-plat}

In this paper we consider a pulse coupled oscillator network with $N=3$ oscillators. We show that there is an open region in the paramater space $(\varepsilon,\tau)$ with families of periodic orbits exhibiting intriguing dynamical behavior. In particular, the periodic orbits are not isolated but for each $(\varepsilon,\tau)$ they fill up a $n\ge2$-dimensional subset in state space, or equivalently, a $n\ge1$-dimensional subset on the Poincar\'e surface of section.

\subsection{Numerical Experiments}
\label{sec/numexp}

We first report the results of numerical experiments for a pulse coupled $3$-oscillator network with delay with parameters $(\varepsilon, \tau)$. Specifically, we numerically compute the orbits of the system starting from a specific class of initial states $\Phi$ on the Poincar\'e surface of section $\theta_3 = 0$. These states are defined by scanning the $(\theta_1,\theta_2)$-space $\mathbb T^2$ and setting $\theta_3 = 0$. As we earlier mentioned this information is not sufficient for determining the dynamics of the system and we also need to know the firing time distances. In this computation, for the oscillators $1$ and $2$ we set
\begin{align}\label{eq/init}
  \Sigma_i = \{ \sigma_i^{(j)} \}= \begin{cases}
    \{ \theta_i \}, &  \text{if $\theta_i \le \tau$} \\
    \emptyset, &  \text{if $\theta_i > \tau$}.
  \end{cases}
\end{align}
Note that this choice of initial states does not exhaustively cover the phase space due to the restrictions imposed on the FTDs. In particular, we could have also considered initial states with more firing moments in $[-\tau,0]$ but our choice is the simplest natural choice and sufficiently reduces the computational time so as to make the computation feasible while allowing to study the system for different parameter values.

We numerically find that all such orbits are \emph{eventually periodic}. There is a time $T_0$ such that for $t \ge T_0$ it holds that $\Phi(t+T) = \Phi(t)$, where $T > 0$ is the eventual orbit \emph{period}. In other words, each initial state converges in finite time to a periodic attractor with period $T$.

In \autoref{fig/all-periodic-points} we show for $(\varepsilon, \tau) = (0.58, 0.58)$ the projection of the periodic attractors to the $(\theta_1,\theta_2)$-space, that is, we show the phase orbits corresponding to the periodic attractors. The figure shows the existence of periodic orbits with \emph{Poincar\'e periods} $T_P \in \{3,4,5\}$. Note that we did not find any attractors with Poincar\'e periods $T_P = 2$ or $T_P \ge 6$ in this computation. Most importantly, we observe that for $(\varepsilon, \tau) = (0.58, 0.58)$ the attractors with Poincar\'e periods $T_P \in \{3,4,5\}$ are not isolated. Projections of periodic attractors with $T_P=3$ appear to fill one-dimensional sets in the $(\theta_1,\theta_2)$-space. Projections of periodic attractors with $T_P=4$ or $T_P=5$ appear to fill one- and two-dimensional sets. In what follows we analytically study the periodic orbits that we numerically observed. We aim to prove that their projections to the $(\theta_1,\theta_2)$-plane fill one- and two-parameter sets and to describe the appearance of these orbits and their properties.

\subsection{Definitions}

To give a systematic description we classify the periodic orbits into equivalence classes. First, we introduce some notation. Let $O$ be a periodic orbit with period $T > 0$ and denote by $P_{i,j}$ the $j$-th pulse received by the $i$-th oscillator in the time interval $[0,T)$. Denote by $n(P_{i,j})$ the \emph{multiplicity} of the pulse $P_{i,j}$, that is, how many simultaneous pulses correspond to $P_{i,j}$.

\begin{figure*}
  \subfloat[Projection of period-3 orbits\label{P3}]%
  {\includegraphics[width=0.32\linewidth]{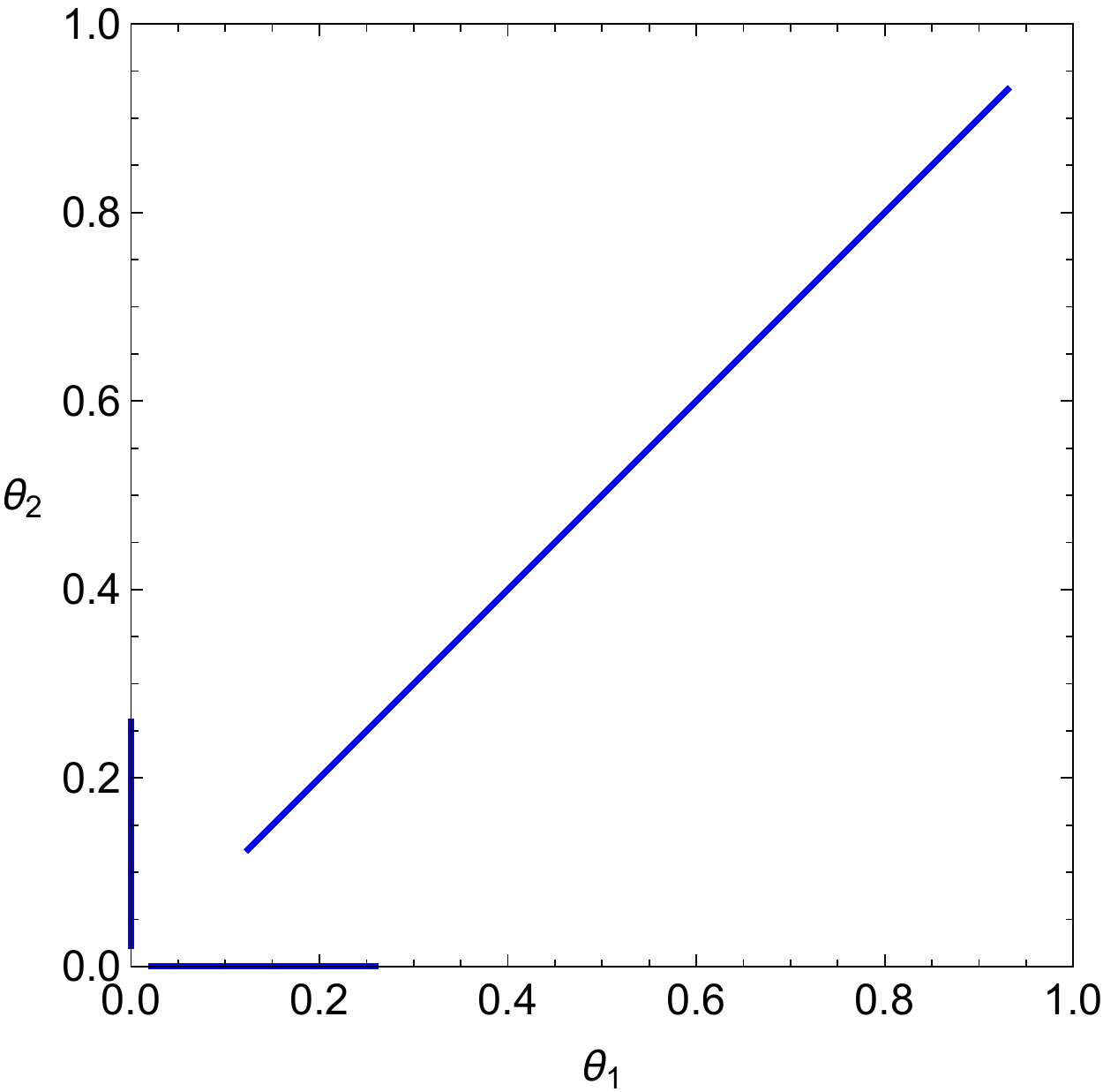}}
\hfill
  \subfloat[Projection of period-4 orbits\label{P4}]%
  {\includegraphics[width=0.32\linewidth]{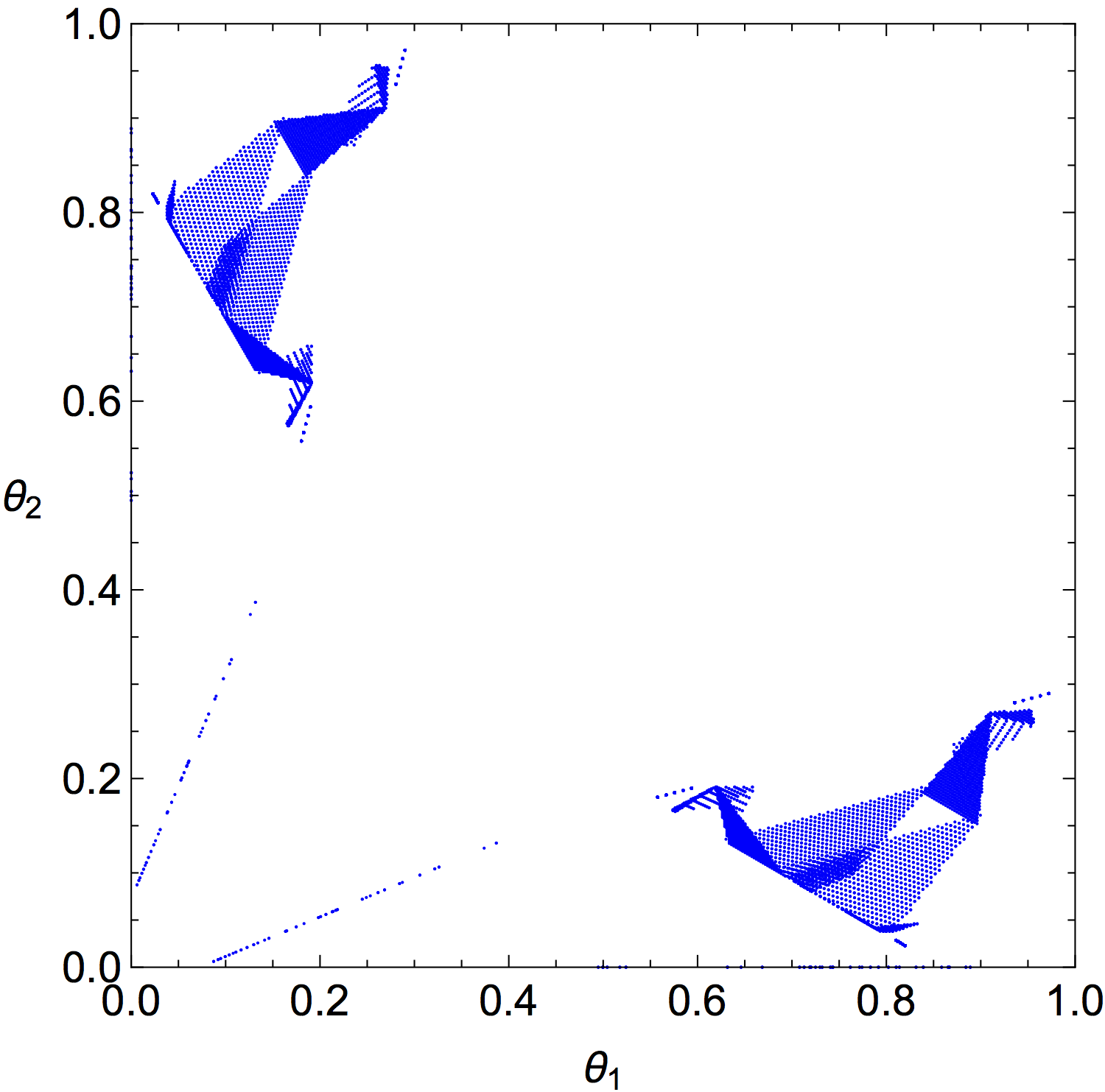}}
\hfill
  \subfloat[Projection of period-5 orbits\label{P5}]%
  {\includegraphics[width=0.32\linewidth]{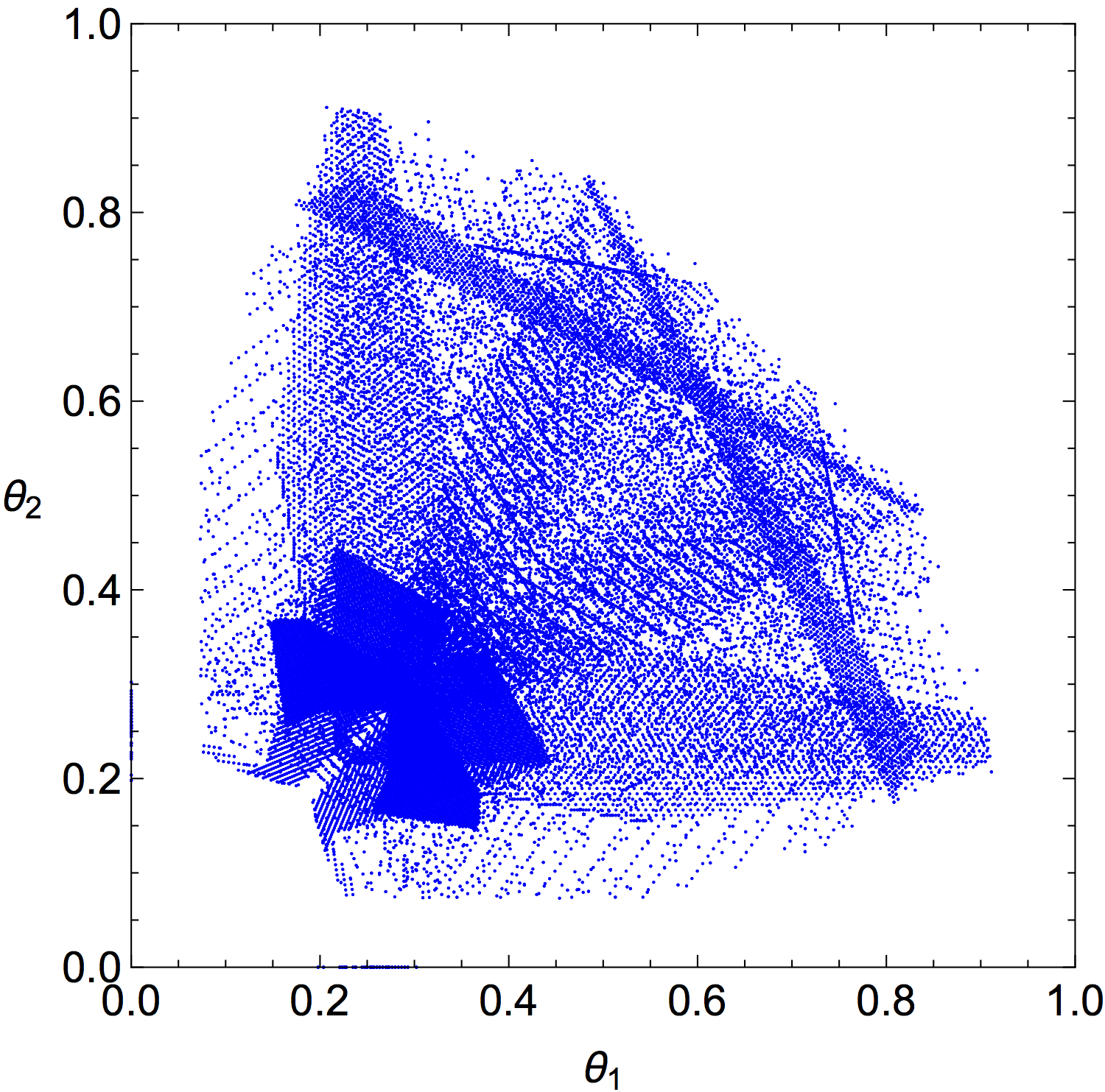}}
\caption{The periodic points of the PCONs of 3-oscillator system under the simulation with parameters $\varepsilon=\tau=0.58$. We choose $(\theta_1(0),\theta_2(0))$ from $[0,1)^2$ with step $10^{-3}$ in each direction. Except for the periodic orbits with period  $T_P \in \{2,3,4\}$, the numerical computation also reveals the existence of the fixed point $(\theta_1,\theta_2)=(0,0)$.}
\label{fig/all-periodic-points}
\end{figure*}

\begin{defn}\label{defn/pulse-eqv}
  Two periodic orbits $O$ and $O'$ are \emph{pulse equivalent} if they have the same periods $T=T' > 0$, the sets $\{ P_{i,j} \}$ and $\{ P'_{i,j} \}$ have the same cardinalities, and $n(P_{i,j}) = n(P'_{i,j})$ for all $i,j$.
\end{defn}

We now define an isochronous region. For this we ask that not only the orbit periods in an isochronous region are the same but the stronger condition that the orbits are pulse equivalent. 

\begin{defn}
  A subset $\mathcal B$ of the state space is an \emph{isochronous region} of period $T$ (or Poincar\'e period $T_P$) if
  \begin{enumerate}[label={(\alph*)}]
  \item all orbits starting in $\mathcal B$ are pulse equivalent with period $T$ (or Poincar\'e period $T_P$),
  \item each orbit starting in $\mathcal B$ stays within $\mathcal B$, and
  \item there is a homeomorphism $S$ between the space of orbits in $\mathcal B$ and an open, connected, subset $\Omega$ of $\mathbb{R}^k$, $k \ge 1$.
  \end{enumerate}
\end{defn}

\begin{remark}
   $\mathcal B$ is required to be invariant under the $\mathbb R_+$ action induced by the dynamics. This allows to define the space of orbits $\mathcal B / \mathbb R_+$ obtained by reducing $\mathcal B$ with respect to the $\mathbb R_+$ action. Note that the requirement that $\mathcal B / \mathbb R_+$ is connected does not imply that $\mathcal B$ is also connected since each periodic orbit in $\mathcal B$ may be disconnected. The requirement that $\dim \mathcal B / \mathbb R_+ \ge 1$ implies that isolated periodic orbits are excluded.
\end{remark}

With these definitions in place, we now turn to the detailed description of one of the isochronous regions that we numerically identified in \autoref{sec/per-plat}.

\section{The isochronous region IR4}
\label{sec/ir4}

In this section we select one of the numerically observed isochronous regions, describe its periodic orbits, and discuss its existence. In \autoref{sec/stability} we consider the dynamical stability of the periodic orbits. Specifically, we focus on the orbits with $T_P=4$ appearing in the lower right corner of \autoref{P4}. We denote the corresponding isochronous region by IR4.

\subsection{Description}

We have verified, analytically and numerically, that all periodic orbits represented by these points can be parameterized by the firing time distances (FTD) $(\sigma_1,\sigma_2,\sigma_3)$ of the three oscillators. We first prove the following slightly more general result which is also useful for determining the stability of the periodic orbits, see \autoref{sec/stability}.

\begin{prop}[Dynamics]\label{prop/dynamics-1}
  Consider the initial state of the pulse coupled $3$-oscillator network on the Poincar\'e surface of section $\theta_3 = 0$, determined by the phases $(\theta_1,\theta_2,0)$ and the firing time distances $(\{\sigma_1\},\{\sigma_2\},\{\sigma_3\})$ satisfying:
  \begin{enumerate}[label={(\alph*)}]
  \item $0 < \sigma_2 < \sigma_1 < \sigma_3 < \tau$,
  \item $H_* < \theta_1 + \tau - \sigma_3 < 1$,
  \item $H_* < H(\theta_2+\tau-\sigma_3)-\sigma_1+\sigma_3 < 1$, 
  \item $H_* < H(\tau-\sigma_1)-\sigma_2+\sigma_1 < 1$,
  \item $H(\sigma_3 - \sigma_2) < 1$.
  \end{enumerate}
  Then the dynamics of the system induces the Poincar\'e map
  \begin{align}
    \begin{aligned}
      G &: (\sigma_1,\sigma_2,\sigma_3;\theta_1,\theta_2) \mapsto \\
      & \quad (\sigma_3-\sigma_2,\sigma_1-\sigma_2,\tau-\sigma_2;H(\sigma_3-\sigma_2),\sigma_1-\sigma_2).
    \end{aligned}
        \label{eq/pmap-1}
  \end{align}
\end{prop}

\begin{proof}
  We use the event sequence representation of the dynamics, see \citep{Broer2008a}. In particular, we denote by $[P,(i_1,\dots,i_k),t]$ a pulse that will be received by the oscillators $i_1,\dots,i_k$ after time $t$. We denote by $[F,i,t]$ the event corresponding to the oscillator $i$ firing after time $t$, further implying that $\theta_i = 1 - t$. The initial condition given by phases $(\theta_1,\theta_2,0)$ and firing time distances $(\{\sigma_1\},\{\sigma_2\},\{\sigma_3\})$ corresponds to the event sequence
  \begin{align*}
    &
      [P,(1,2),\tau-\sigma_3],\,
      [P,(2,3),\tau-\sigma_1],\,
      [P,(1,3),\tau-\sigma_2],
    \\ &
      [P,(1,2),\tau];\,
      [F,1,1-\theta_1],\,
      [F,2,1-\theta_2],\,
      [F,3,1].
  \end{align*}
  Note that we write pulse events separately from fire events, keeping the time ordering in each of the subsets. In particular, this implies that $0 < \sigma_2 < \sigma_1 < \sigma_3 < \tau$ and that $0 < \theta_2 < \theta_1 < 1$. 

  The inequality $\theta_1 + \tau - \sigma_3 < 1$ implies that the first pulse event will be processed first. Then the next event sequences will be
  \begin{align*}
    \overset{1}{\longrightarrow}\quad
    &
      [P,(1,2),0],\,
      [P,(2,3),\sigma_3-\sigma_1],\,
      [P,(1,3),\sigma_3-\sigma_2],\,
    \\
    &
      [P,(1,2),\sigma_3];\,
      [F,1,1-\theta_1-\tau+\sigma_3],\,
    \\
    &
      [F,2,1-\theta_2-\tau+\sigma_3],\,
      [F,3,1-\tau+\sigma_3]
    \\
    \overset{2}{\longrightarrow}\quad
    & [P,(2,3),\sigma_3-\sigma_1],\,
      [P,(1,3),\sigma_3-\sigma_2],\,
      [P,(1,2),\sigma_3];\,
    \\
    & [F,1,0],\, 
      [F,2,1-H(\theta_2+\tau-\sigma_3)],\,
      [F,3,1-\tau+\sigma_3].
  \end{align*}
  Here we used the assumptions that $\theta_1 + \tau - \sigma_3 > H_*$ and $\theta_2 + \tau - \sigma_3 < H_*$. The next event sequence is
  \begin{align*}
    \overset{3}{\longrightarrow}\quad
    & [P,(2,3),\sigma_3-\sigma_1],\,
      [P,(1,3),\sigma_3-\sigma_2],\,
    \\
    &
      [P,(1,2),\sigma_3],\,
      [P,(2,3),\tau];\,
      [F,2,1-H(\theta_2+\tau-\sigma_3)],\,
    \\
    &
      [F,3,1-\tau+\sigma_3],\,
      [F,1,1]
    \\
    \overset{4}{\longrightarrow}\quad
    & [P,(2,3),0],\,
      [P,(1,3),\sigma_1-\sigma_2],\,
      [P,(1,2),\sigma_1],\,
    \\
    &
      [P,(2,3),\tau+\sigma_1-\sigma_3];\,
    \\
    &
      [F,2,1-H(\theta_2+\tau-\sigma_3)+\sigma_1-\sigma_3],\,
    \\
    &
      [F,3,1-\tau+\sigma_1],\,
      [F,1,1+\sigma_1-\sigma_3]
    \\
    \overset{5}{\longrightarrow}\quad
    & [P,(1,3),\sigma_1-\sigma_2],\,
      [P,(1,2),\sigma_1],\,
    \\
    &
      [P,(2,3),\tau+\sigma_1-\sigma_3];\,
      [F,2,0],\,
    \\
    &
      [F,3,1-H(\tau-\sigma_1)],\,
      [F,1,1+\sigma_1-\sigma_3]
    \\
    \overset{6}{\longrightarrow}\quad
    & [P,(1,3),\sigma_1-\sigma_2],\,
      [P,(1,2),\sigma_1],\,
    \\
    &
      [P,(2,3),\tau+\sigma_1-\sigma_3],\,
      [P,(1,3),\tau];\,
    \\
    &
      [F,3,1-H(\tau-\sigma_1)],\,
      [F,1,1+\sigma_1-\sigma_3],\,
      [F,2,1].
  \end{align*}
  The inequality $H(\theta_2+\tau-\sigma_3)-\sigma_1+\sigma_3 <1$ implies again that the first pulse event was processed first and then $H(\theta_2+\tau-\sigma_3)-\sigma_1+\sigma_3 > H_*$ that oscillator $2$ fires. Moreover, the assumption $\tau - \sigma_1 < H_*$ ensures that oscillator $3$ does not fire.
  The next event sequence is
  \begin{align*}
    \overset{7}{\longrightarrow}\quad
    & [P,(1,3),0],\,
      [P,(1,2),\sigma_2],\,
      [P,(2,3),\tau+\sigma_2-\sigma_3],\,
    \\
    &
      [P,(1,3),\tau+\sigma_2-\sigma_1];\,
    \\
    &
      [F,3,1-H(\tau-\sigma_1)+\sigma_2-\sigma_1],\,
    \\
    &
      [F,1,1+\sigma_2-\sigma_3],\,
      [F,2,1+\sigma_2-\sigma_1]
    \\
    \overset{8}{\longrightarrow}\quad
    & [P,(1,2),\sigma_2],\,
      [P,(2,3),\tau+\sigma_2-\sigma_3],\,
    \\
    &
      [P,(1,3),\tau+\sigma_2-\sigma_1];\,
      [F,3,0],\,
    \\
    &
      [F,1,1-H(\sigma_3-\sigma_2)],\,
      [F,2,1+\sigma_2-\sigma_1].
  \end{align*}
  Here, by the assumptions $H_* < H(\tau-\sigma_1)-\sigma_2+\sigma_1 < 1$ and $H(\sigma_3 - \sigma_2) + \sigma_2< 1$, we have
  \begin{align*}
    \overset{9}{\longrightarrow}\quad
    & [P,(1,2),\sigma_2],\,
      [P,(2,3),\tau+\sigma_2-\sigma_3],\,
    \\
    &
      [P,(1,3),\tau+\sigma_2-\sigma_1],\,
      [P,(1,2),\tau]
      ;\,
    \\
    & [F,1,1-H(\sigma_3-\sigma_2)],\,
      [F,2,1+\sigma_2-\sigma_1],\,
      [F,3,1],
  \end{align*}
  thus proving the statement.
\end{proof}

Let $\Omega_{\varepsilon,\tau}$ be the subset of the $(\sigma_1,\sigma_2,\sigma_3)$-space defined by the relations
\begin{subequations}
  \begin{align}\label{eq/ineq-pulse-class-A}
    \begin{aligned}
      & 0 < \sigma_2 < \sigma_1 < \sigma_3 < \tau, \\
      & H_* \le F_k(\bm{\sigma}; \tau) \le 1, \ \text{$k=1,2,3,4$},
    \end{aligned}
  \end{align}
  where
  \begin{align}\label{eq/def-F}
    \begin{aligned}
      F_1(\bm{\sigma};\tau) &:= H(\sigma_1)+\tau-\sigma_3, \\
      F_2(\bm{\sigma};\tau) &:= H(\tau-\sigma_3+\sigma_2)+\sigma_3-\sigma_1, \\
      F_3(\bm{\sigma};\tau) &:= H(\tau-\sigma_1)+\sigma_1-\sigma_2, \\
      F_4(\bm{\sigma};\tau) &:= H(\sigma_3-\sigma_2)+\sigma_2, \\
    \end{aligned}
  \end{align}
\end{subequations}
and
\begin{align*}
  H_* &= H_1^{-1}(1) = \frac{e^b-e^{b\hat{\varepsilon}}}{(e^b-1)e^{b\hat{\varepsilon}}}.
\end{align*}
Moreover, define the map
\begin{align}
\begin{aligned}
  S&: (\sigma_1, \sigma_2, \sigma_3)
  \mapsto  (\theta_1, \theta_2, \theta_3; \{\sigma_1\}, \{\sigma_2\}, \{\sigma_3\}) \\
  &= (H(\sigma_1), \sigma_2, 0; \{\sigma_1\}, \{\sigma_2\}, \{\sigma_3\}),
\end{aligned}
\end{align}
from $\Omega_{\varepsilon,\tau}$ to the space of initial conditions of the PCON. Then we prove the following statement.

\begin{prop}\label{prop/dynamics-2}
  Consider the initial state of the pulse coupled $3$-oscillator network on the Poincar\'e surface of section $\theta_3 = 0$, given by $S(\bm{\sigma})$ for $\bm{\sigma} \in \Omega_{\varepsilon,\tau}$. Then the map
  \begin{align}
    \label{eq/pmap-g-sigma}
    g : (\, \sigma_1,\, \sigma_2,\, \sigma_3\,) \mapsto (\,\sigma_3-\sigma_2,\, \sigma_1-\sigma_2,\, \tau-\sigma_2\,)
  \end{align}
  has the following properties:
  \begin{enumerate}[label={(\alph*)}]
  \item $g(\Omega_{\varepsilon,\tau}) = \Omega_{\varepsilon,\tau}$;
  \item $G(S(\bm{\sigma})) = S(g(\bm{\sigma}))$ for all $\bm{\sigma} \in \Omega_{\varepsilon,\tau}$, where $G$ is the Poincar\'e map \eqref{eq/pmap-1}.
  \end{enumerate}
\end{prop}

\begin{proof}
  First, one easily checks that if $\bm{\sigma} \in \Omega_{\varepsilon,\tau}$ then $g(\bm{\sigma}) \in \Omega_{\varepsilon,\tau}$ and vice versa. Then, note that if $\bm{\sigma} \in \Omega_{\varepsilon,\tau}$ then $S(\bm{\sigma})$ satisfies the conditions of Proposition~\ref{prop/dynamics-1}. This implies
  \begin{align*}
    G(S(\bm{\sigma}))
    & = (\sigma_3-\sigma_2,\sigma_1-\sigma_2,\tau-\sigma_2;\; H(\sigma_3-\sigma_2),\sigma_1-\sigma_2) \\
    & = S(g(\bm{\sigma})).
  \end{align*}
\end{proof}

Proposition~\ref{prop/dynamics-2} shows that $S$ intertwines the map $g$ on $\Omega_{\varepsilon,\tau}$ with the Poincar\'e map $G$. We then have the following description of the dynamics in $\Omega_{\varepsilon,\tau}$.

\begin{prop}\label{prop/dynamics-3}
  The map $g$ on $\Omega_{\varepsilon,\tau}$ has period $4$, that is, $g^4(\bm{\sigma}) = \bm{\sigma}$ for all $\bm{\sigma} \in \Omega_{\varepsilon,\tau}$. The point $\bm{\sigma}_* := (\tau/2,\tau/4,3\tau/4) \in \Omega_{\varepsilon,\tau}$ is a fixed point of $g$, and points
  $\bm{\sigma} \in \Omega_{\varepsilon,\tau}$ along the line parameterized by $\bm{\sigma} = \bm{\sigma}_* + t (0,1,1)$, $t \in \mathbb R$, are period-$2$ points of $g$.
\end{prop}

\begin{proof}
  The proof of the statement is a straightforward computation. Nevertheless, it is more enlightening to proceed in a different way. Let
\begin{align*}
  \bm{\sigma} = \bm{\sigma}_* + \bm{s},
\end{align*}
where $\bm{s} = (s_1,s_2,s_3)$. In terms of $\bm{s}$, $g$ becomes the linear map
\begin{align*}
  g(\bm{s}) = L \bm{s},
\end{align*}
where
\begin{align*}
  L =
  \begin{pmatrix}
    0 & -1 & 1 \\
    1 & -1 & 0 \\
    0 & -1 & 0
  \end{pmatrix}.
\end{align*}
Clearly, $\bm{s}=0$ is the only fixed point of $L$. One checks that $L^2$ acts as rotation by $\pi$ about the line $\bm{s} = t(0,1,1)$, $t \in \mathbb R$. Therefore, $L^2$ leaves this line invariant, and $L^4$ is the identity.
\end{proof}

\begin{remark}
  Proposition \ref{prop/dynamics-2} implies that the Poincar\'e map $G$ has Poincar\'e period $T_P = 4$ for each $S(\bm{\sigma})$, $\bm{\sigma} \in \Omega_{\varepsilon,\tau}$. The evolution of the phases of the $3$ oscillators for such orbits is depicted in \autoref{fig/p4p4} and the detailed dynamics is given in \autoref{tbl/pulse-class-A}. The set $\Omega_{\varepsilon,\tau}$ also gives rise to periodic orbits with smaller minimal orbit period than $T = 3\tau$. In particular, there is a line in $\Omega_{\varepsilon,\tau}$ given by $(\sigma_1,\sigma_2,\sigma_3) = (\tau/2,\sigma_2,\tau/2+\sigma_2)$ for which all points give rise to period $T=3\tau/2$ orbits ($T_P=2$), see \autoref{fig/p2p4}. One point along this line, having $\sigma_2 = \tau/4$ gives rise to a period $T=3\tau/4$ orbit ($T_P=1$), see \autoref{fig/p1p4}.
\end{remark}

\begin{figure}[tp]
  \subfloat[\label{fig/p4p4}]{\includegraphics[width=\linewidth]{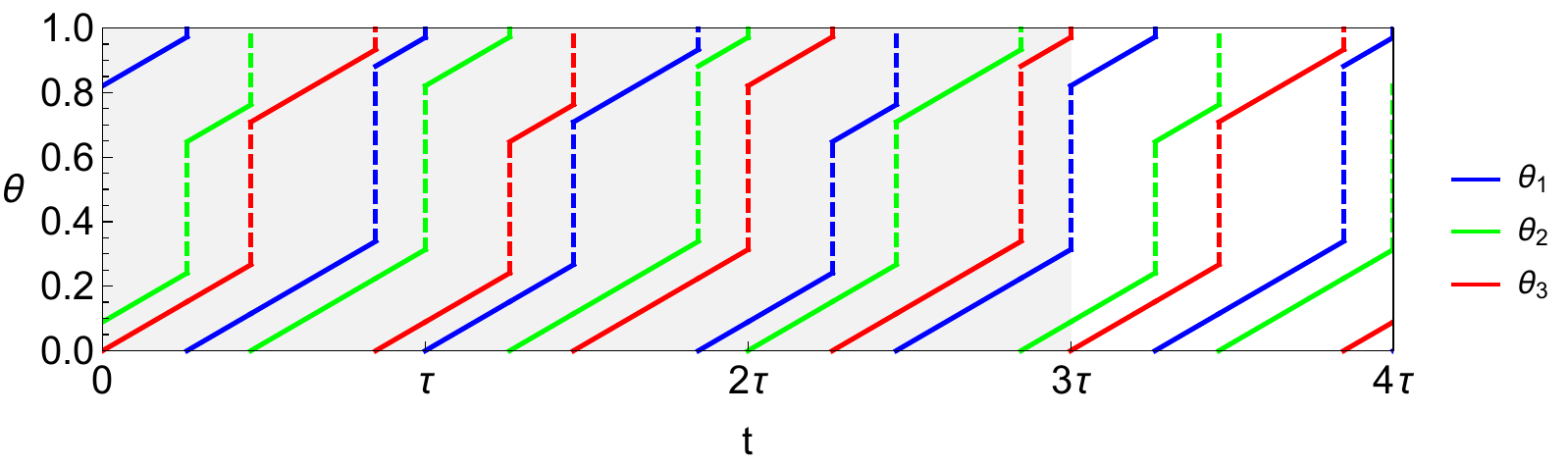}} \\
  \subfloat[\label{fig/p2p4}]{\includegraphics[width=\linewidth]{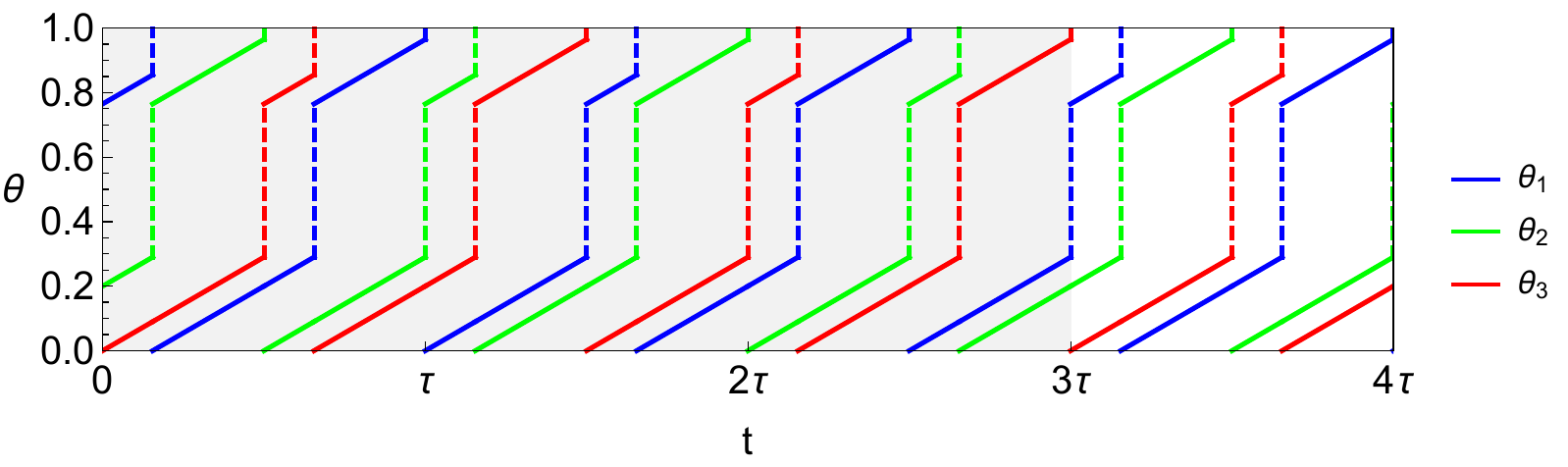}} \\
  \subfloat[\label{fig/p1p4}]{\includegraphics[width=\linewidth]{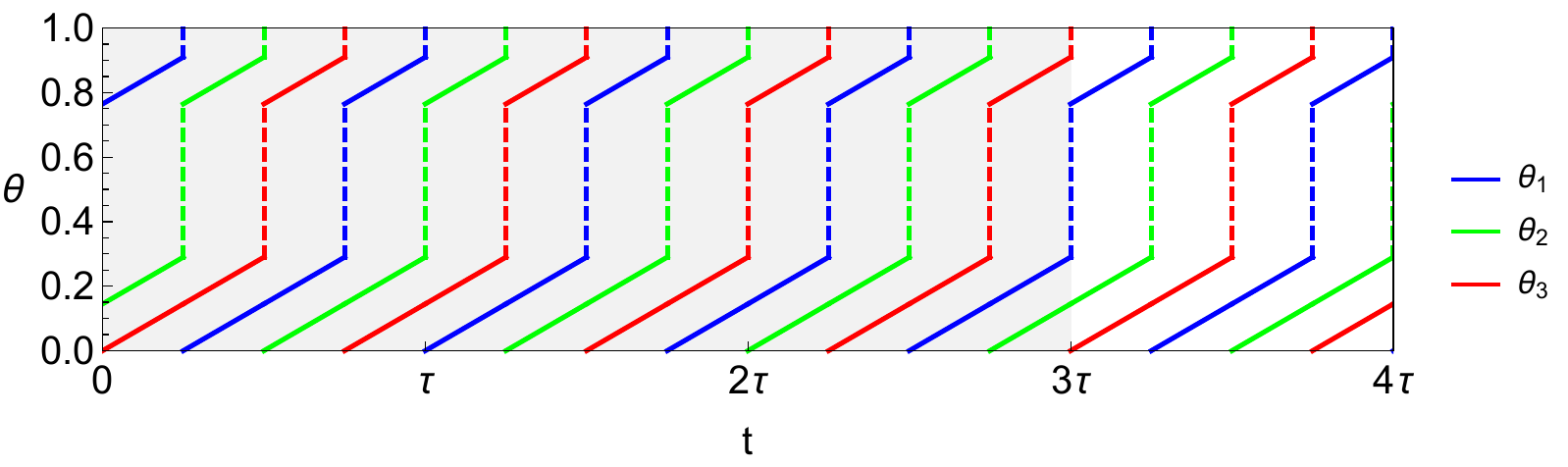}}
  \caption{Evolution of the phases of the $3$ oscillators for different orbits in the isochronous region IR4. The shaded region represents one period of the generic orbit, that is, $T=3\tau$; dashed vertical lines represent phase jumps induced by the reception of pulses. From top to bottom: (a) Generic orbit, $T = 3\tau$ and $T_P = 4$; (b) $\bm{\sigma} = (\tau/2,\sigma_2,\tau/2+\sigma_2)$, $\sigma_2 \ne \tau/4$, giving $T = 3\tau/2$ and $T_P = 2$; (c) $\bm{\sigma} = \bm{\sigma}_* = (\tau/2,\tau/4,3\tau/4)$, giving $T = 3\tau/4$ and $T_P = 1$.}
\end{figure}

\begin{table*}[tbp]
\centering
\begin{tabular}{llll}
\toprule
time       & $O_1[\{\sigma_1 \}, \theta_1$]       & $O_2[\{\sigma_2 \}, \theta_2$]        & $O_3[\{\sigma_3 \}, \theta_3$]   \\
\colrule
$-\sigma_3$      & $[- , -]$       & $[-, -]$       & $[-, F]$      \\
$-\sigma_1$   & $[- , F]$  & $[-, -]$           & $[-, -]$     \\
$-\sigma_2$           & $[-, -]$       & $[-, F]$ & $[-, -]$       \\
0   &  $[\sigma_1, H(\sigma_1)]$       & $[\sigma_2, \sigma_2]$      &$ [\sigma_3, 0]$      \\
\colrule
$\tau-\sigma_3$      & $[\tau-\sigma_3+\sigma_1, 0]$       & $[\tau-\sigma_3+\sigma_2, H(\tau-\sigma_3+\sigma_2)]$       & $[\tau-\sigma_3, \tau-\sigma_3]$      \\
$\tau-\sigma_1$      & $[\sigma_3-\sigma_1, \sigma_3-\sigma_1]$       & $[\tau-\sigma_1+\sigma_2, 0]$       & $[\tau-\sigma_1, H(\tau-\sigma_1)]$      \\
$\tau-\sigma_2$      & $[\sigma_3-\sigma_2, H(\sigma_3-\sigma_2)]$   & $[\sigma_1-\sigma_2, \sigma_1-\sigma_2]$   & $[\tau-\sigma_2, 0]$      \\
\colrule
$\tau$      & $[ \sigma_3, 0]$       & $[\sigma_1, H(\sigma_1)]$       & $[\sigma_2, \sigma_2]$      \\
$2\tau-\sigma_3$      & $[\tau-\sigma_3, \tau-\sigma_3]$       & $[\tau-\sigma_3+\sigma_1, 0]$       & $[\tau-\sigma_3+\sigma_2, H(\tau-\sigma_3+\sigma_2)]$      \\
$2\tau-\sigma_1$      & $[\tau-\sigma_1, H(\tau-\sigma_1)]$       & $[\sigma_3-\sigma_1, \sigma_3-\sigma_1]$       & $[\tau-\sigma_1+\sigma_2, 0]$      \\
\colrule
$2\tau-\sigma_2$      & $[\tau-\sigma_2, 0]$     & $[\sigma_3-\sigma_2, H(\sigma_3-\sigma_2)]$   & $[\sigma_1-\sigma_2, \sigma_1-\sigma_2]$          \\
$2\tau$      & $[\sigma_2, \sigma_2]$     & $[ \sigma_3, 0]$       & $[\sigma_1, H(\sigma_1)]$         \\
$3\tau-\sigma_3$      & $[\tau-\sigma_3+\sigma_2, H(\tau-\sigma_3+\sigma_2)]$     & $[\tau-\sigma_3, \tau-\sigma_3]$       & $[\tau-\sigma_3+\sigma_1, 0]$         \\
\colrule
$3\tau-\sigma_1$      & $[\tau-\sigma_1+\sigma_2, 0]$     & $[\tau-\sigma_1, H(\tau-\sigma_1)]$       & $[\sigma_3-\sigma_1, \sigma_3-\sigma_1]$         \\
$3\tau-\sigma_2$      & $[\sigma_1-\sigma_2, \sigma_1-\sigma_2]$      & $[\tau-\sigma_2, 0]$     & $[\sigma_3-\sigma_2, H(\sigma_3-\sigma_2)]$       \\
$3\tau$      & $[\sigma_1, H(\sigma_1)]$       & $[\sigma_2, \sigma_2]$     & $[ \sigma_3, 0]$          \\
\botrule
\end{tabular}
\caption{Dynamics for a periodic orbit in the isochronous region IR4. The periodic orbits in this isochronous region have Poincar\'e period $T_P = 4$ and period $T = 3\tau$ in the full phase space. In the time interval $(-\tau,0)$, the only useful information is the firing moments, so we use `$-$' in the table to represent the ``useless'' information and `$F$' to represent an oscillator fired at the given moment.}
\label{tbl/pulse-class-A}
\end{table*}

\subsection{Existence}

\begin{figure}[tp]
  \subfloat[\label{p4por-A}]%
  {\includegraphics[width=0.48\linewidth]{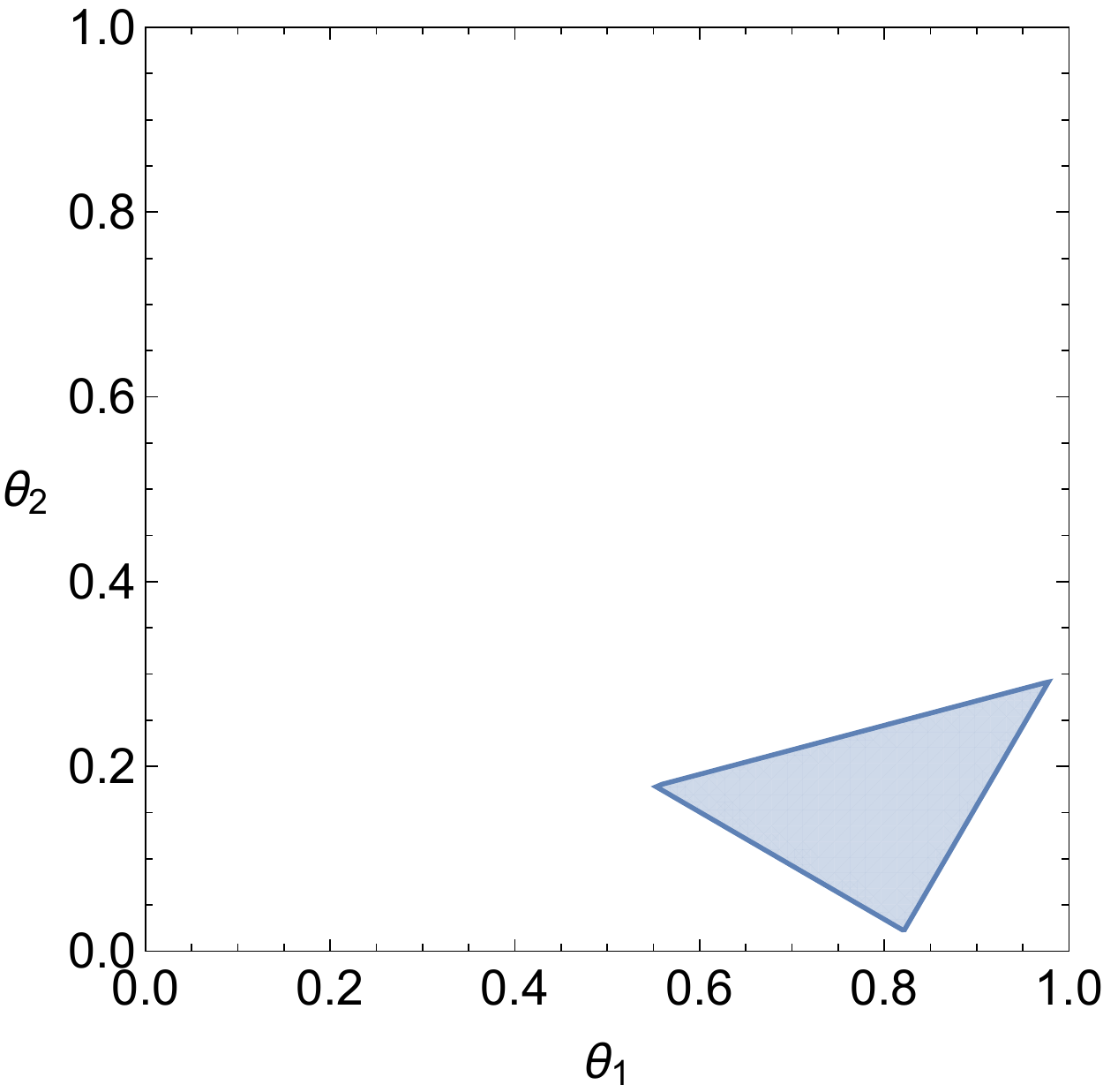}}
  \hfill
  \subfloat[\label{p4por-Omega}]%
  {\includegraphics[width=0.48\linewidth]{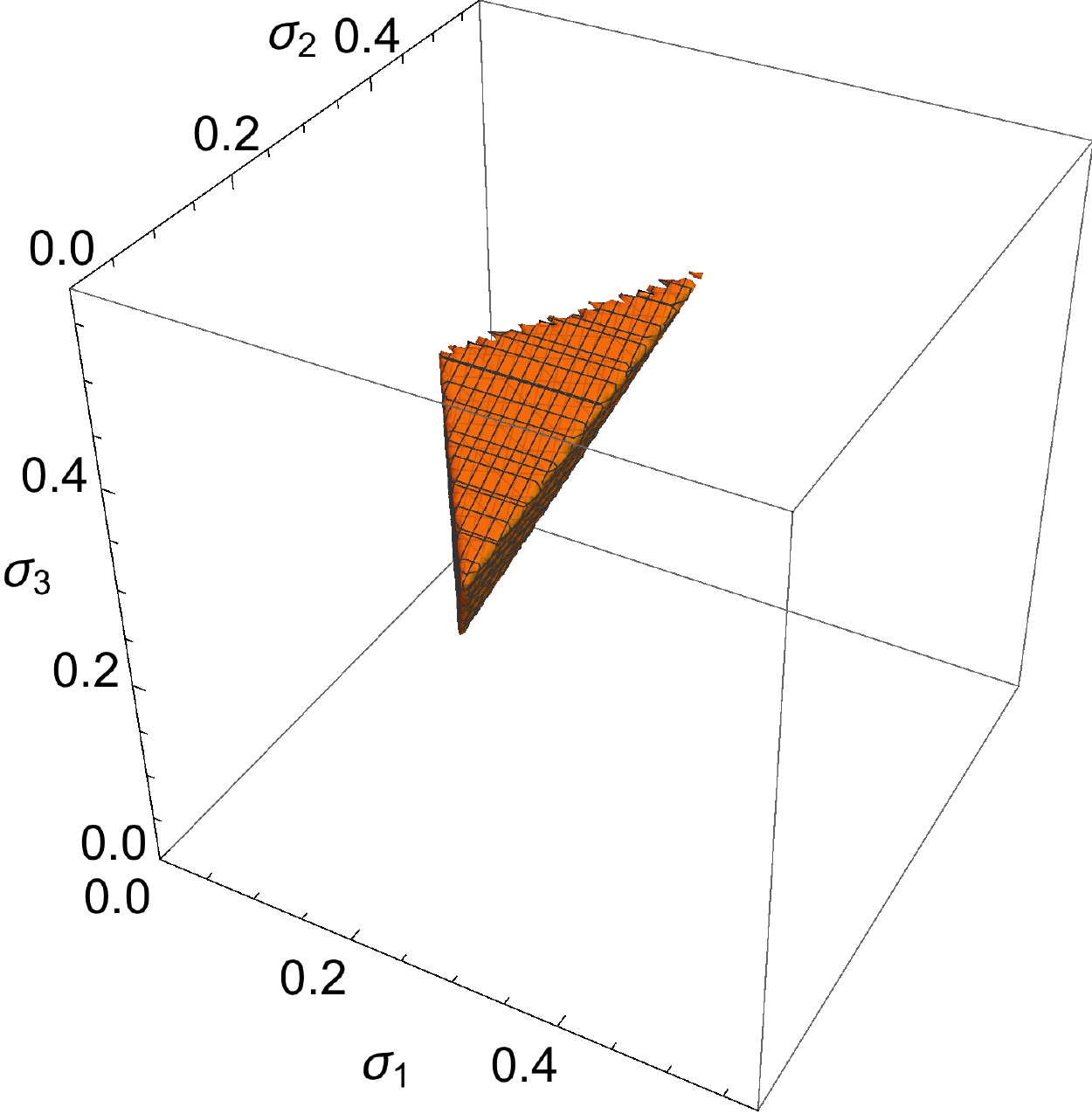}}
  \\
  \subfloat[\label{fig/existOmega}]%
  {\includegraphics[width=0.48\linewidth]{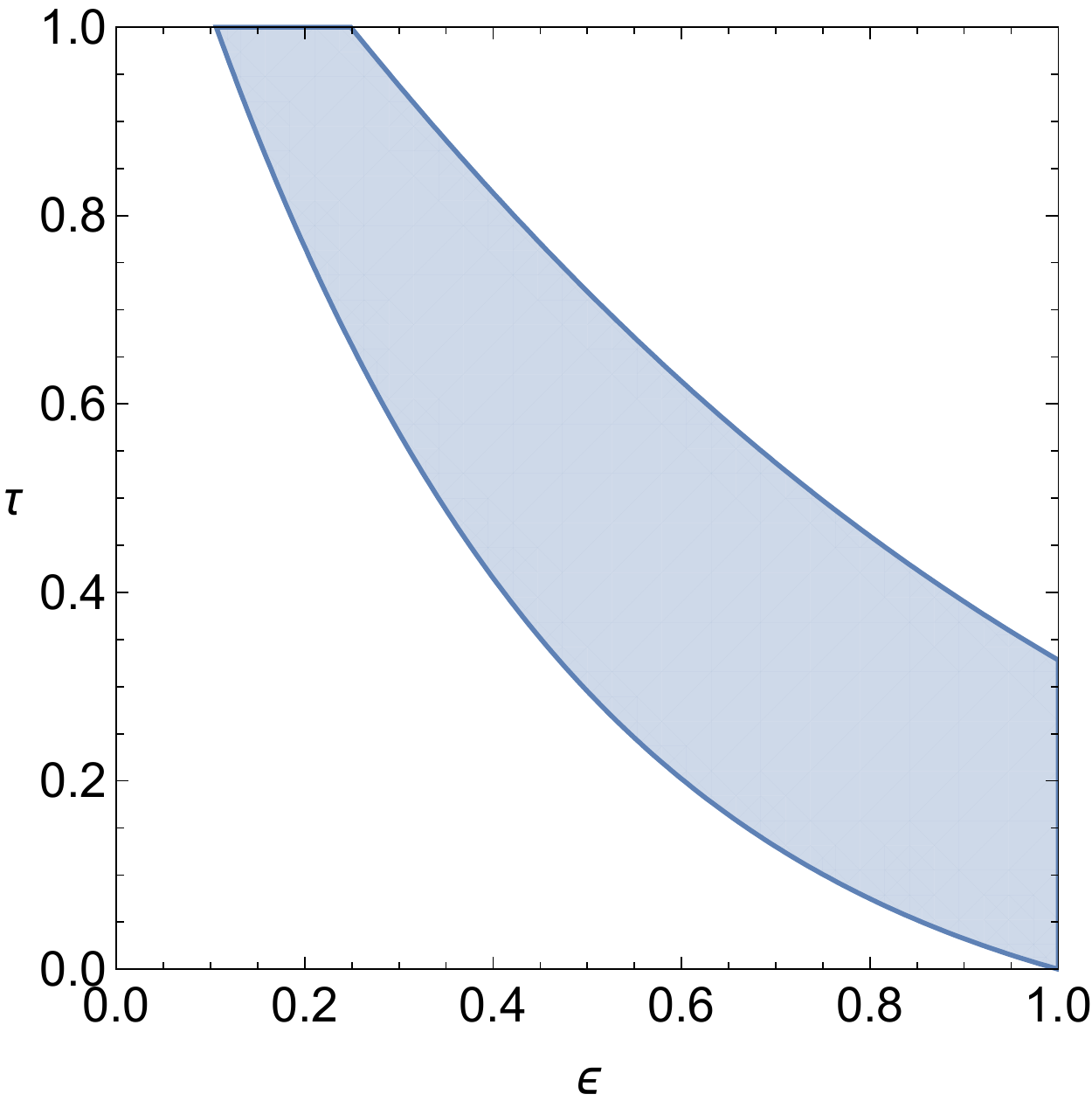}}
  \hfill
  \subfloat[\label{fig/volOmega}]%
  {\includegraphics[width=0.48\linewidth]{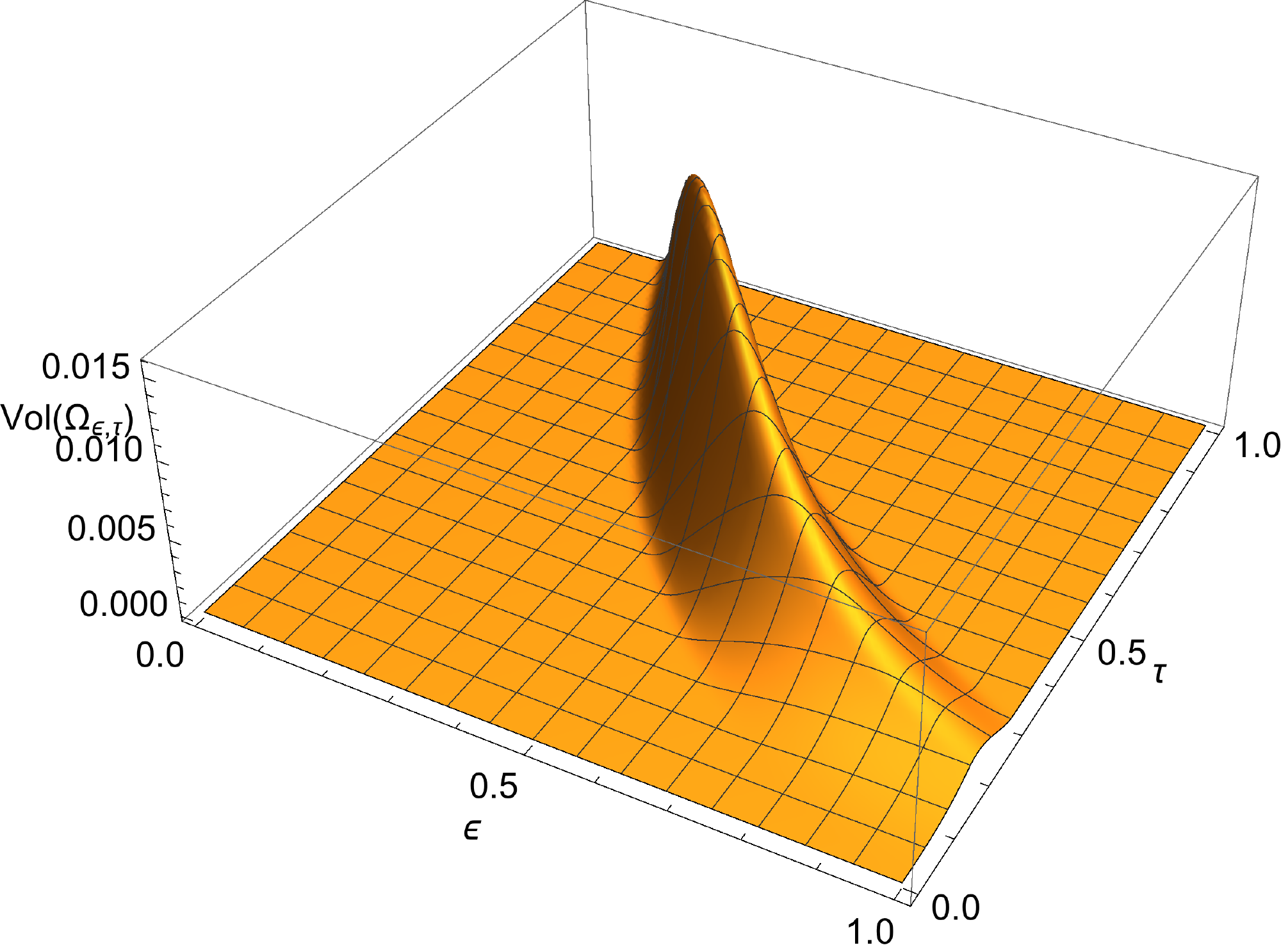}}%
  \caption{Sets $\mathcal{A}_{\varepsilon,\tau}$ and $\Omega_{\varepsilon,\tau}$ for IR4. (a) $\mathcal{A}_{\varepsilon,\tau}$; cf.~\autoref{fig/all-periodic-points}, for $(\varepsilon,\tau)=(0.58,0.58)$. (b) $\Omega_{\varepsilon,\tau}$ for $(\varepsilon,\tau)=(0.58,0.58)$. (c) The shaded region represents the subset of parameter space $(\varepsilon,\tau)$, given by Eq.~\eqref{eq/master-ineq}, for which $\Omega_{\epsilon,\tau}$ is non-empty. (d) Volume of $\Omega_{\varepsilon,\tau}$.}
\end{figure}

Let $\wOmega_{\varepsilon,\tau} = S(\Omega_{\varepsilon,\tau} )$ be the embedding of $\Omega_{\varepsilon,\tau}$ in the $(\bm{\theta};\bm{\sigma})$-space. Moreover, let $\mathcal{A}_{\varepsilon,\tau} = \mathrm{pr}_\theta(\wOmega_{\varepsilon,\tau})$, where $\mathrm{pr}_\theta: \mathbb R^6 \to \mathbb R^2$ is the projection to the $(\theta_1,\theta_2)$-plane. The set $\mathcal{A}_{\varepsilon,\tau}$ is depicted in \autoref{p4por-A} for $(\varepsilon,\tau)=(0.58,0.58)$. One can check that $\Omega_{\varepsilon,\tau}$ and $\mathcal{A}_{\varepsilon,\tau}$ have non-empty interior for $(\varepsilon, \tau) = (0.58, 0.58)$ and that each point in $\wOmega_{\varepsilon,\tau}$ is the initial condition of a periodic orbit in IR4 of \autoref{tbl/pulse-class-A} with period $T = 3\tau$ and Poincar\'e period $T_P = 4$ . Therefore, $\wOmega_{\varepsilon,\tau}$ is a periodic plateau.

\begin{prop}
  The isochronous region IR4 of \autoref{tbl/pulse-class-A} exists in the subset of the parameter space $(\varepsilon, \tau)$ given by
  \begin{align} \label{eq/master-ineq}
    H_* \le H\left( \frac{\tau}{2} \right) + \frac{\tau}{4} \le 1,
  \end{align}
  see \autoref{fig/existOmega}.
\end{prop}

\begin{proof}
Note that
\begin{align*}
  \frac14 \sum_{k=1}^4 F_k(\bm{\sigma};\tau) = H\left(\frac{\tau}{2}\right) + \frac{\tau}{4}.
\end{align*}
This implies that Eq.~\eqref{eq/master-ineq} is a necessary condition for Eq.~\eqref{eq/ineq-pulse-class-A} to hold. We show that if Eq.~\eqref{eq/master-ineq} holds then $\Omega_{\varepsilon,\tau}$ contains a non-empty open subset. Consider the point
\begin{align*}
  \bm{\sigma}_* = \left(\frac{\tau}{2}, \frac{\tau}{4}, \frac{3\tau}{4} \right).
\end{align*}
Then $\bm{\sigma}_* \in \Omega_{\varepsilon,\tau}$ if and only if Eq.~\eqref{eq/master-ineq} holds, since in this case we have $F_k(\bm{\sigma}_*;\tau) = H(\tau/2) + \tau/4$, for $k=1,\dots,4$. Therefore, when Eq.~\eqref{eq/master-ineq} holds, $\Omega_{\varepsilon,\tau} \ne \emptyset$. Moreover, when the strict form of Eq.~\eqref{eq/master-ineq} holds, there is an open neighborhood $U$ of $D_\tau$ in $\bm{\sigma}$-space such that $U \subset \Omega_{\varepsilon,\tau}$.
\end{proof}

\begin{figure}[htbp]
  \centering
  \includegraphics[width=\linewidth]{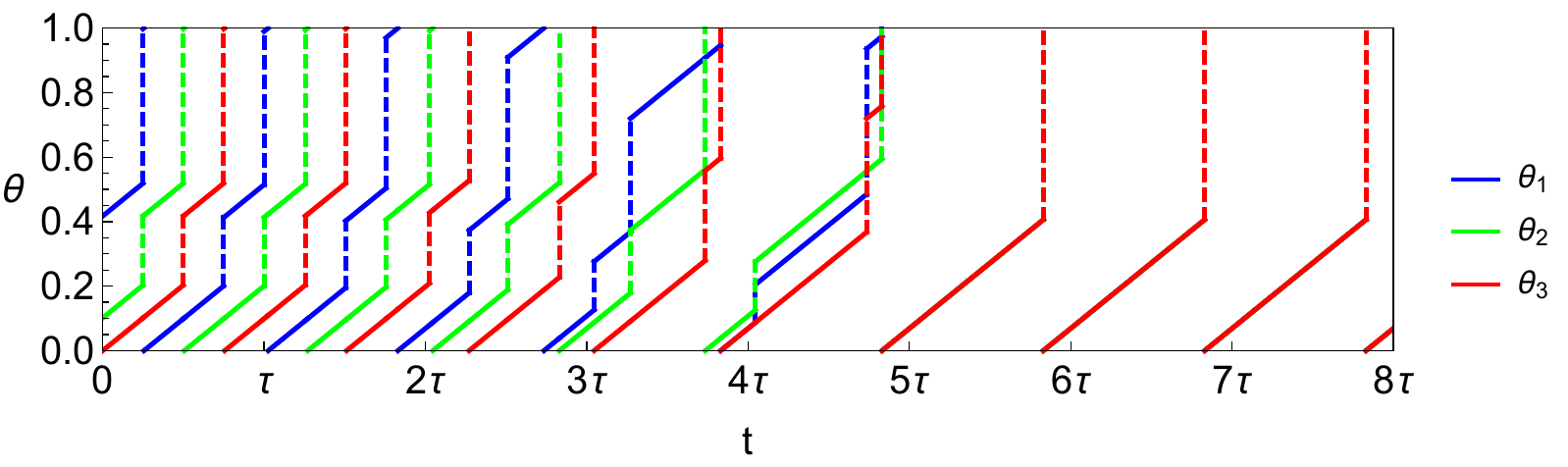}
  \caption{Evolution of the phases of the 3 oscillators with $\bm{\sigma} = \bm{\sigma}_* = (\tau/2,\tau/4,3\tau/4)$ outside the region.}
  \label{fig/p1out}
\end{figure}

\autoref{fig/volOmega} shows the volume of $\Omega_{\varepsilon,\tau}$ for $(\varepsilon,\tau) \in [0,1]^2$. The volume is computed using the Mathematica function \texttt{Volume}.

\begin{remark}
  If $H(\frac{\tau}{2})+\frac{\tau}{4} = H_*$ or $H(\frac{\tau}{2})+\frac{\tau}{4}=1$, the inequalities \eqref{eq/ineq-pulse-class-A} are satisfied only by the point $\bm{\sigma}_*$. \autoref{fig/p1out} shows the phases for an orbit starting from the point $\bm{\sigma}_*$ when $(\varepsilon,\tau)$ moves outside the region of existence of IR4. In that case the dynamics converges in short time to a stable periodic orbit with $T_P=1$.
\end{remark}

\begin{figure}[htbp]
  \centering
  \subfloat[$(\varepsilon,\tau) = (0.58,0.58)$.\label{fig/comp058}]%
  {\includegraphics[width=0.48\linewidth]{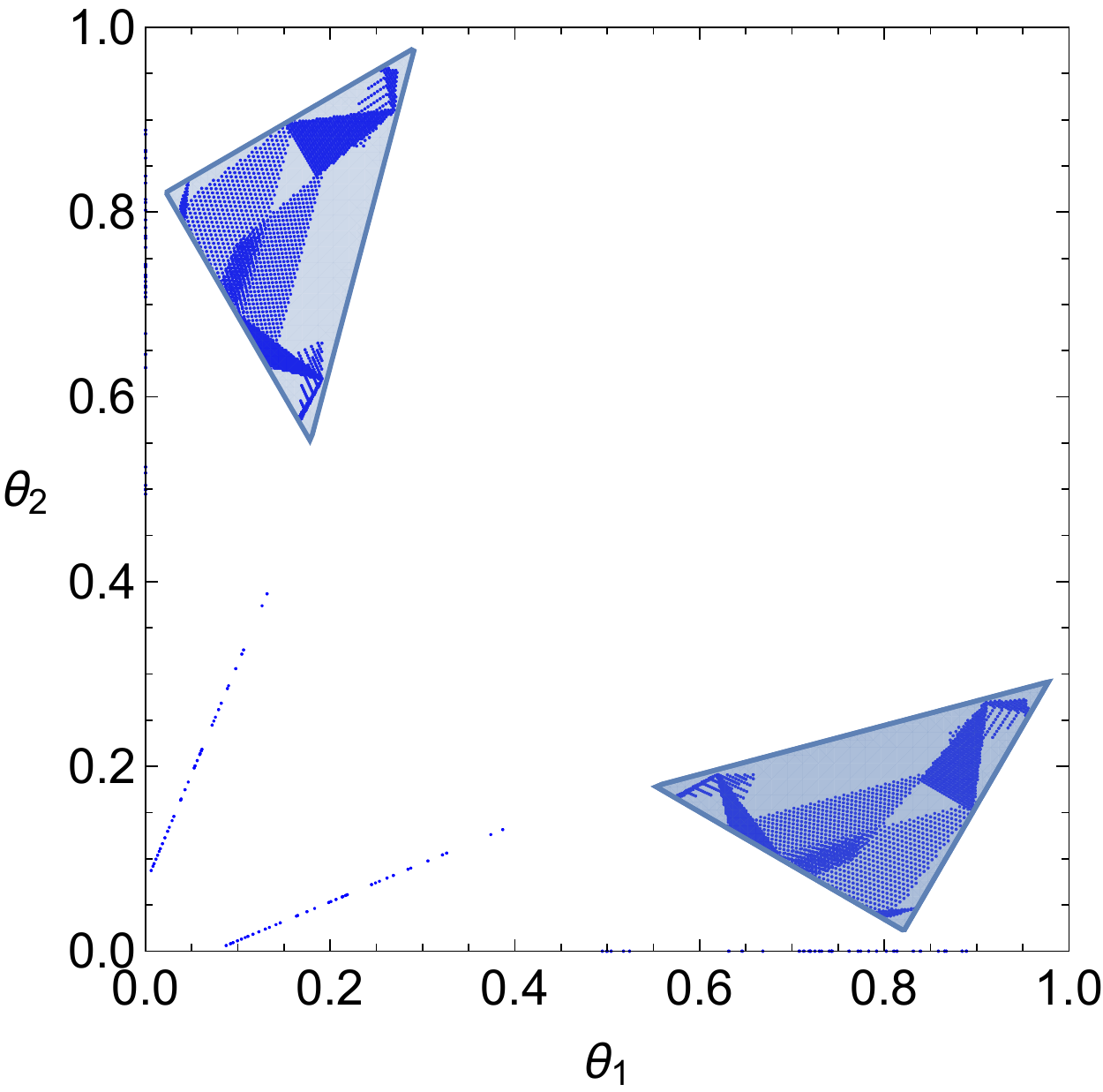}}
  \hfill
  \subfloat[$(\varepsilon,\tau) = (0.45,0.45)$.\label{fig/comp045}]%
  {\includegraphics[width=0.48\linewidth]{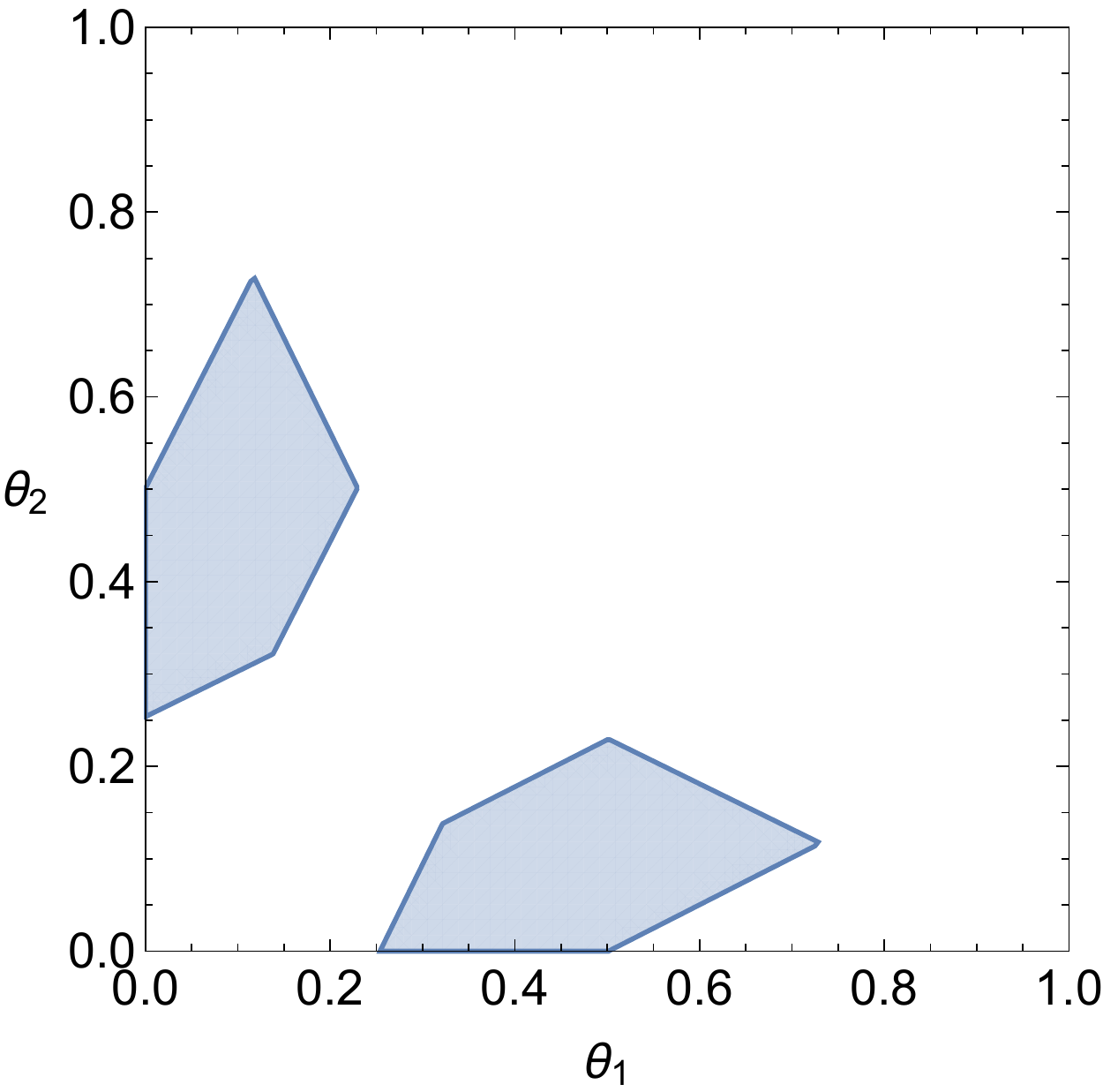}}
  \caption{Comparison between the analytically obtained $\mathcal A_{\varepsilon,\tau}$ and the numerically computed orbits in IR4 in our initial numerical computations.}
\end{figure}

\autoref{fig/comp058} compares the analytically obtained $\mathcal A_{\varepsilon,\tau}$ for $(\varepsilon,\tau) = (0.58,0.58)$ to the numerical results discussed in \autoref{sec/numexp}. We note that the numerically obtained orbits cover only part of $\mathcal A_{\varepsilon,\tau}$. This can be explained by the fact that the space of initial conditions that we scanned in our numerical experiments does not include the periodic orbits in IR4. Some of the orbits in IR4 are periodic attractors for our numerical initial conditions but others are not accessible. This effect is much more pronounced for $(\varepsilon,\tau) = (0.45,0.45)$ as is shown in \autoref{fig/comp045}. In this case our initial numerical experiments did not reveal the existence of any orbits in IR4. Nevertheless, in this case $\Omega_{\varepsilon,\tau}$ is non-empty and subsequent numerical experiments with different initial conditions allowed us to numerically find orbits in IR4.

\subsection{Stability}
\label{sec/stability}

In this section we consider the stability of the periodic orbits in IR4. We show that, for $\bm{\sigma} \in \Omega_{\varepsilon,\tau}$, small changes in $\bm{\theta}$ lead to the same periodic orbit while small changes in $\bm{\sigma}$ lead to a nearby periodic orbit in the same pulse equivalence class. In particular, we have the following result.

\begin{prop}\label{localbasin}
  Let $(\bm{\theta},\bm{\sigma}) = S(\bm\sigma)$, $\bm{\sigma} \in \Omega_{\varepsilon,\tau}$, and denote by $Y(\bm{\sigma})$ the corresponding periodic orbit in IR4. Then for $\Delta\bm{\theta}$ and $\Delta\bm{\sigma}$ sufficiently small, the orbit with initial condition $(\bm{\theta}+\Delta\bm{\theta},\bm{\sigma}+\Delta\bm{\sigma})$ converges in one iteration of the Poincar\'e map to the periodic orbit $Y(\bm{\sigma}+\Delta\bm{\sigma})$.
\end{prop}

\begin{proof}
  This statement is a straightforward consequence of Proposition \ref{prop/dynamics-1}. Since $\Omega_{\varepsilon,\tau}$ is open in $\mathbb R^3$, given $\bm{\sigma} \in \Omega_{\varepsilon,\tau}$, there is an open neighborhood $U \ni \bm{\sigma}$ with $U \subseteq \Omega_{\varepsilon,\tau}$. Therefore, for $\Delta\bm{\sigma}$ small enough we have $\bm{\sigma} + \Delta \bm{\sigma} \in \Omega_{\varepsilon,\tau}$. Therefore, $(\bm{\theta}',\bm{\sigma}+\Delta\bm{\sigma}) = S(\bm{\sigma} + \Delta \bm{\sigma})$ satisfies the conditions of Proposition \ref{prop/dynamics-1}. This implies that for sufficiently small $\Delta\bm{\theta}$ we have that $(\bm{\theta}'+\Delta\bm{\theta}',\bm{\sigma}+\Delta\bm{\sigma})$ also satisfies the conditions of Proposition \ref{prop/dynamics-1} giving convergence to $Y(\bm{\sigma}+\Delta\bm{\sigma})$. Finally, we note that $\Delta\bm{\theta}'$ can be made sufficiently small by making $\Delta\bm{\theta}$ sufficiently small because of the continuity of the map $S$.
\end{proof}

\section{Other isochronous regions}
\label{sec/other}

The isochronous region IR4 is not the only such region that appears in the system under consideration here. Here we briefly report on two other such regions.

\subsection{The isochronous region IR3}

The isochronous region IR3 consists of periodic orbits with Poincar\'e period $T_P=3$. For orbits in IR3, two of the oscillators have the same phase. This implies that the projection of orbits in IR3 to the $(\theta_1,\theta_2)$-plane lies either on one of the axes or along the diagonal.
Let $\Omega_{\varepsilon,\tau}$ be the subset of the $(\sigma_1,\sigma_3)$-space defined by the relations
\begin{subequations}
  \begin{align}\label{eq/ineq-pulse-class-B}
    \begin{aligned}
      & 0 < \sigma_1 < \sigma_3 < \tau, \\
      & H_* \le F_k(\bm{\sigma}; \tau) \le 1, \ \text{$k=1,2,3$}, \\
      & H_{**} \le F_k(\bm{\sigma}; \tau) \le 1, \ \text{$k=4,5,6$},
    \end{aligned}
  \end{align}
  where
  \begin{align}\label{eq/def-F-B}
    \begin{aligned}
      F_1(\bm{\sigma};\tau) &:= H(\sigma_2)+\tau-\sigma, \\
      F_2(\bm{\sigma};\tau) &:= H(\tau-\sigma_3)+\sigma_3-\sigma_2, \\
      F_3(\bm{\sigma};\tau) &:= H(\sigma_3-\sigma_2)+\sigma_2, \\
      F_4(\bm{\sigma};\tau) &:= \sigma_3, \\
      F_5(\bm{\sigma};\tau) &:= \tau-\sigma_2, \\
      F_6(\bm{\sigma};\tau) &:= \tau+\sigma_2-\sigma_3,
    \end{aligned}
  \end{align}
\end{subequations}
and
\begin{align*}
  H_{**} &= H_2^{-1}(1) = \frac{e^{b}-e^{2b\hat{\varepsilon}}}{(e^{b}-1)e^{2b\hat{\varepsilon}}}.
\end{align*}
Then we consider in state space the set $S(\Omega_{\varepsilon,\tau})$ where $S$ is given by
\begin{align}
\begin{aligned}
  S&: (\sigma_1, \sigma_3)
  \mapsto  (\theta_1, \theta_2, \theta_3; \{\sigma_1\}, \{\sigma_2\}, \{\sigma_3\}) \\
  &= (\sigma_1, \sigma_1, 0; \{\sigma_1\}, \{\sigma_1\}, \{\sigma_3\}).
\end{aligned}
\end{align}
The region $\Omega_{\varepsilon,\tau}$ and the projection $\mathcal A_{\varepsilon,\tau}$ of $S(\Omega_{\varepsilon,\tau})$ on the $(\theta_1,\theta_2)$-plane are shown in \autoref{fig/ir3-sigma-theta}.

\begin{figure}[tp]
  \centering 
  \subfloat[$\Omega_{\varepsilon,\tau}$.\label{fig/ir3-sigma}]%
  {\includegraphics[width=0.48\linewidth]{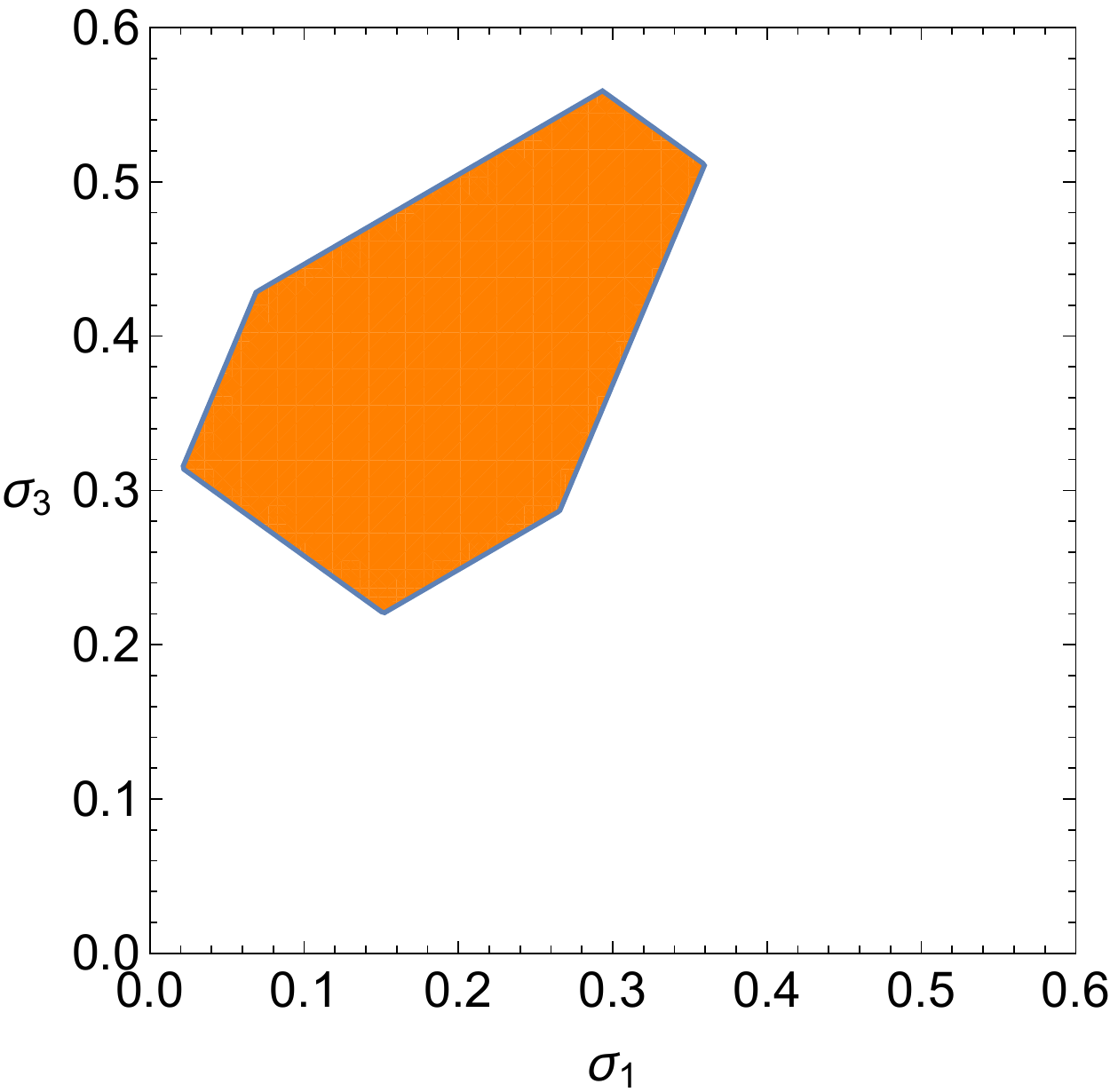}}%
  \hfill%
  \subfloat[$\mathcal A_{\varepsilon,\tau}$.\label{fig/ir3-theta}]%
  {\includegraphics[width=0.48\linewidth]{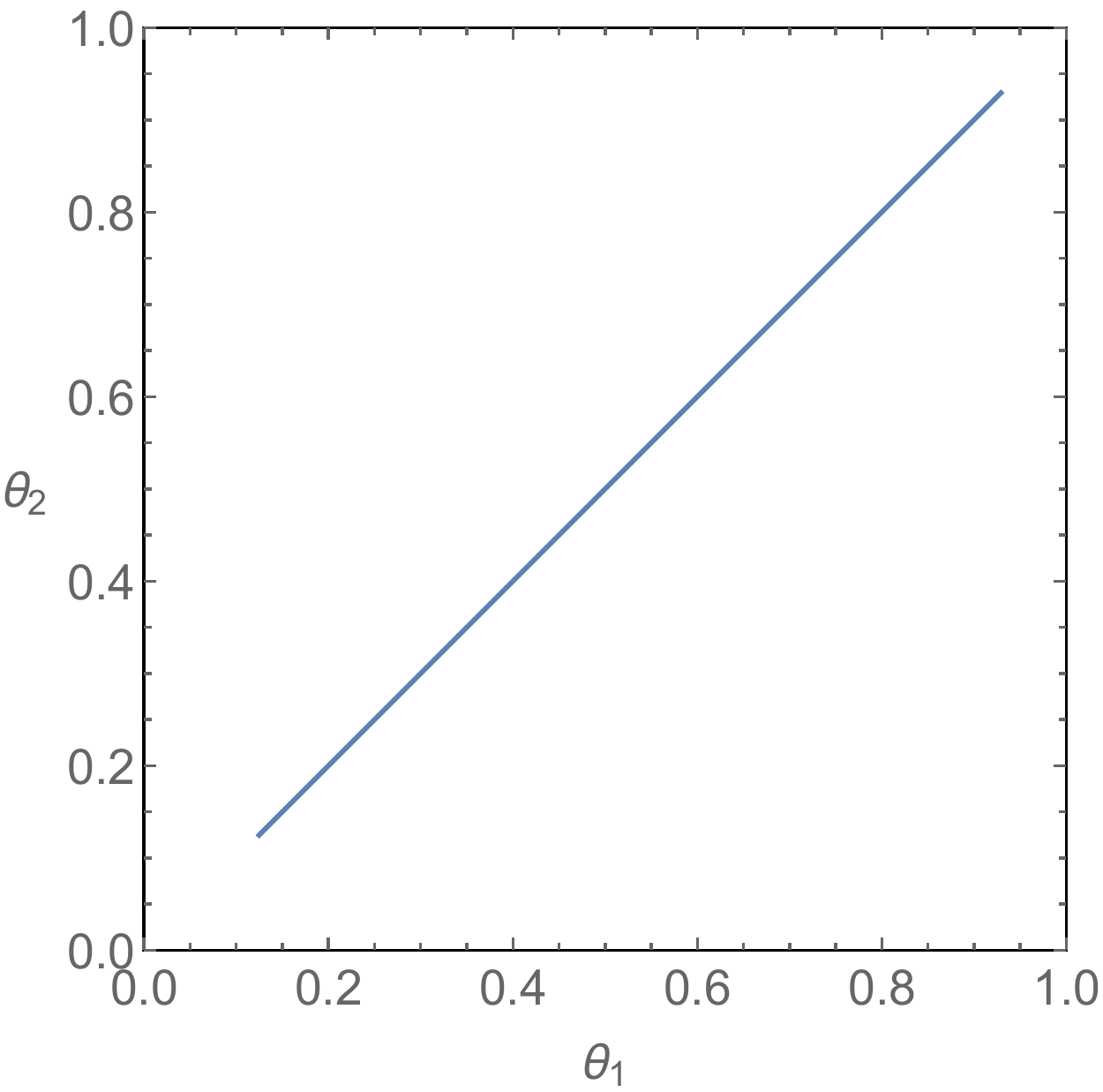}}%
  \caption{The sets $\Omega_{\varepsilon,\tau}$ and $\mathcal A_{\varepsilon,\tau}$ for IR3 and parameter values $(\varepsilon,\tau)=(0.58,0.58)$.}
  \label{fig/ir3-sigma-theta}
\end{figure}

\begin{figure}[tp]
  \centering
  \includegraphics[width=\linewidth]{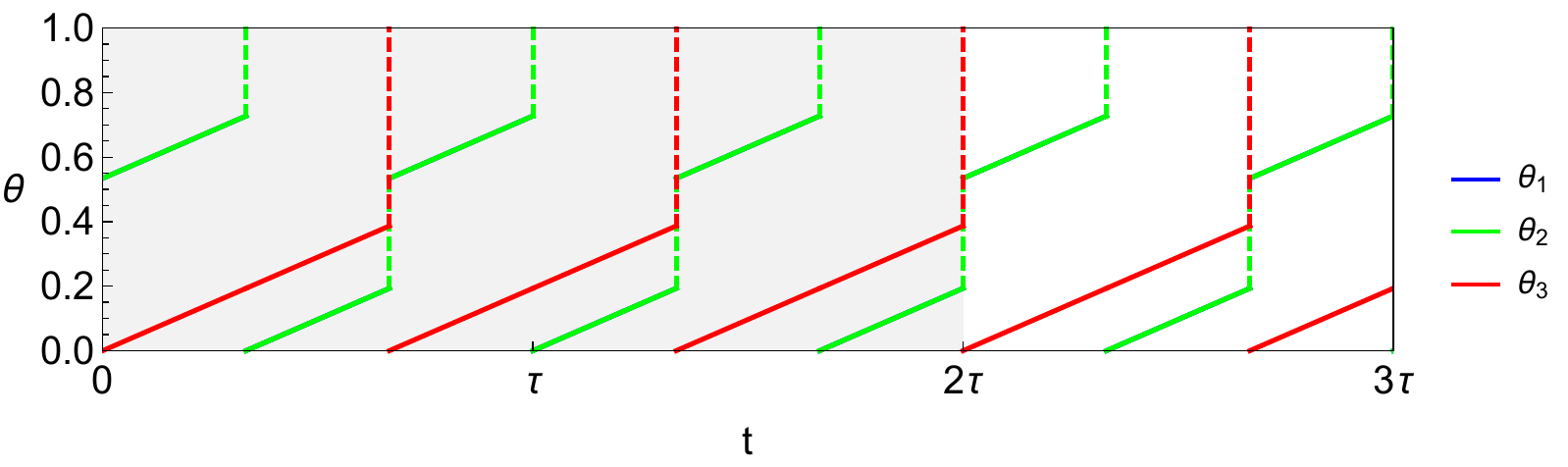}
  \caption{Phase evolution for the period $1$ point in IR3. Note that oscillators $1$ and $2$ are synchronized, that is, $\theta_1 = \theta_2$.}
  \label{fig/ir3-center-dynamics}
\end{figure}

\begin{figure}[tp]
  \centering
  \subfloat[Parameter region for IR3.\label{fig/ir3-exist}]%
  {\includegraphics[width=0.48\linewidth]{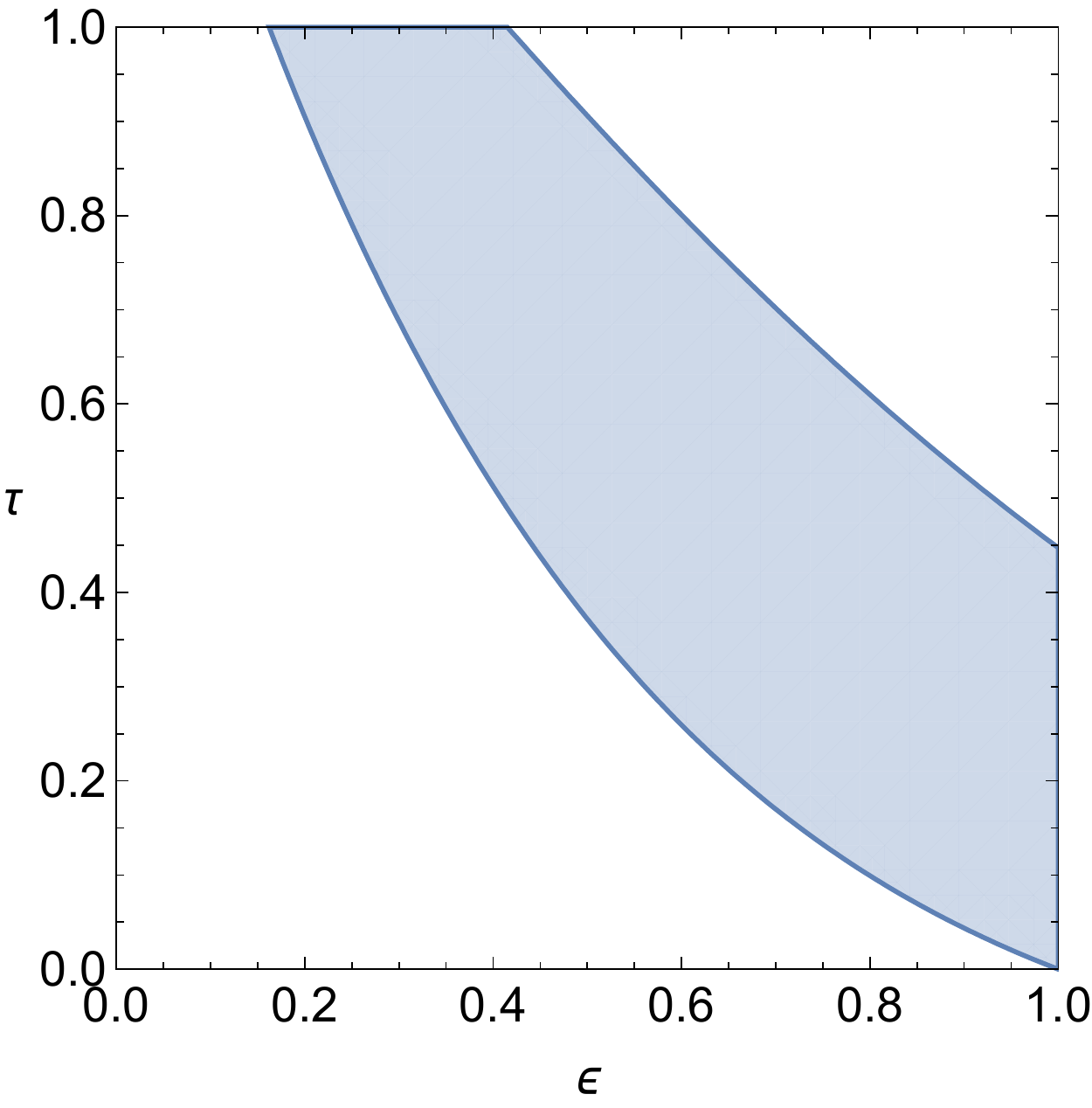}}
  \hfill
  \subfloat[Parameter region for IR5.\label{fig/ir5-exist}]%
  {\includegraphics[width=0.48\linewidth]{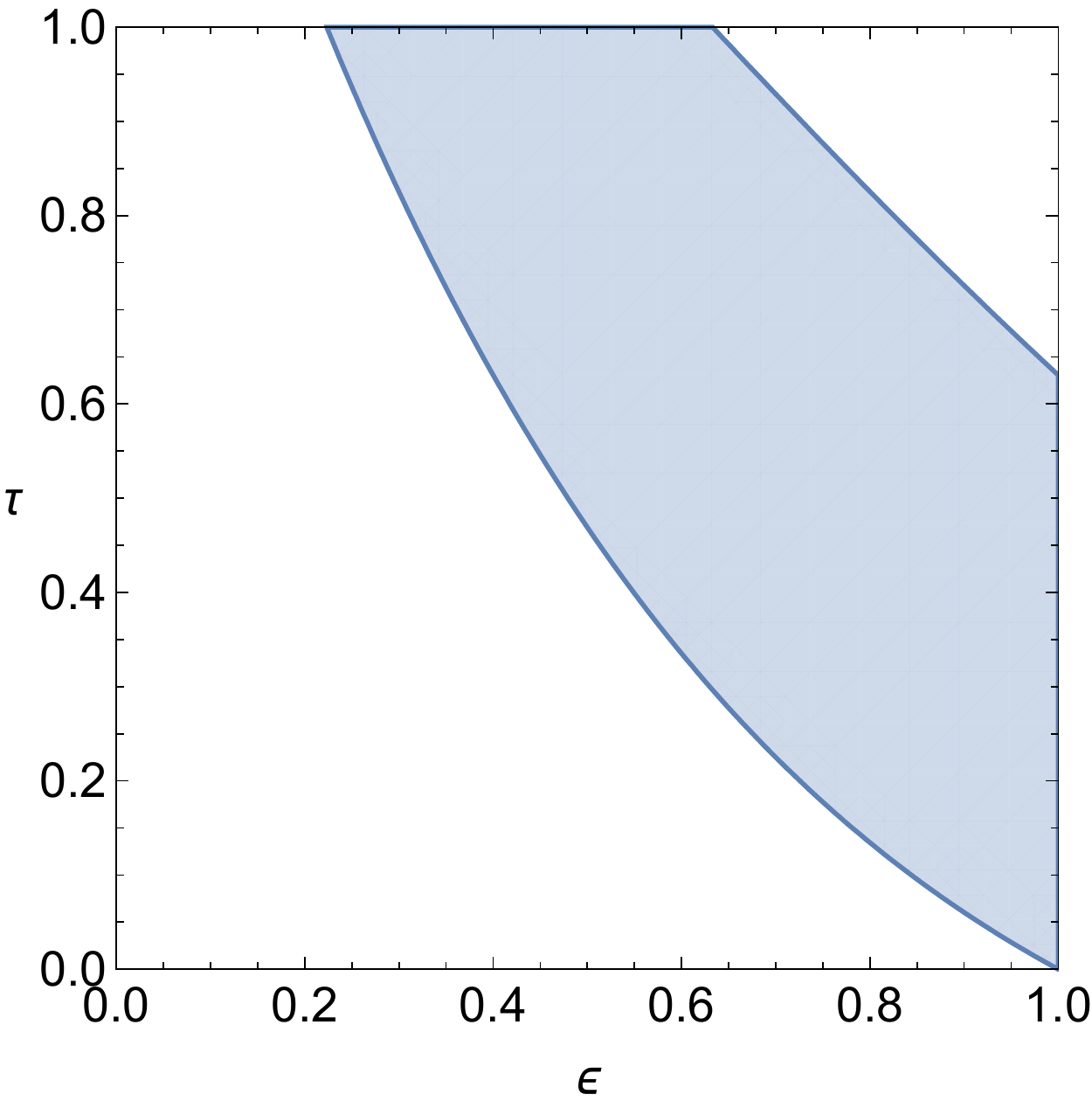}}
  \caption{Subsets of the parameter space $(\varepsilon,\tau)$ for which the network has isochronous regions IR3 and IR5.}
\end{figure}

Using similar arguments as in the analysis of IR4 we find that the point
\begin{align*}
  \bm{\sigma}_* = \left( \sigma_1, \sigma_3\right)
           = \left(\frac{\tau}{3}, \frac{2\tau}{3} \right),
\end{align*}
gives a periodic orbit with Poincar\'e period $T_P=1$. Its phase evolution is shown in \autoref{fig/ir3-center-dynamics}. Moreover, we find that this occurs for
\begin{align} \label{eq/master-ineq3}
  H_* \le H\left( \frac{\tau}{3} \right) + \frac{\tau}{3} \le 1,
\end{align}
thus giving the subset of the parameter space $(\varepsilon, \tau)$ for which IR3 exists, see \autoref{fig/ir3-exist}.

\subsection{The isochronous region IR5}

\begin{figure}[tp]
  \centering
  \includegraphics[width=\linewidth]{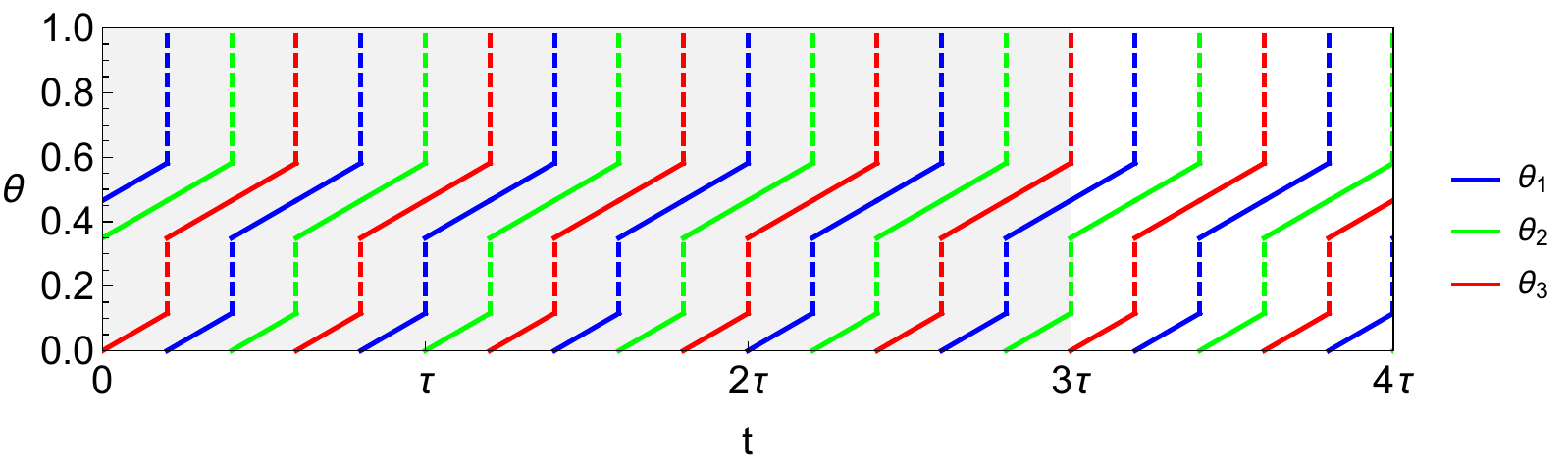}
  \caption{Phase dynamics for the period $1$ point in IR5.}
  \label{fig/ir5-center-dynamics}
\end{figure}

\begin{figure}[tp]
  \centering
  \subfloat[$\mathcal A_{\varepsilon,\tau}$.\label{fig/ir5-theta}]%
  {\includegraphics[width=0.48\linewidth]{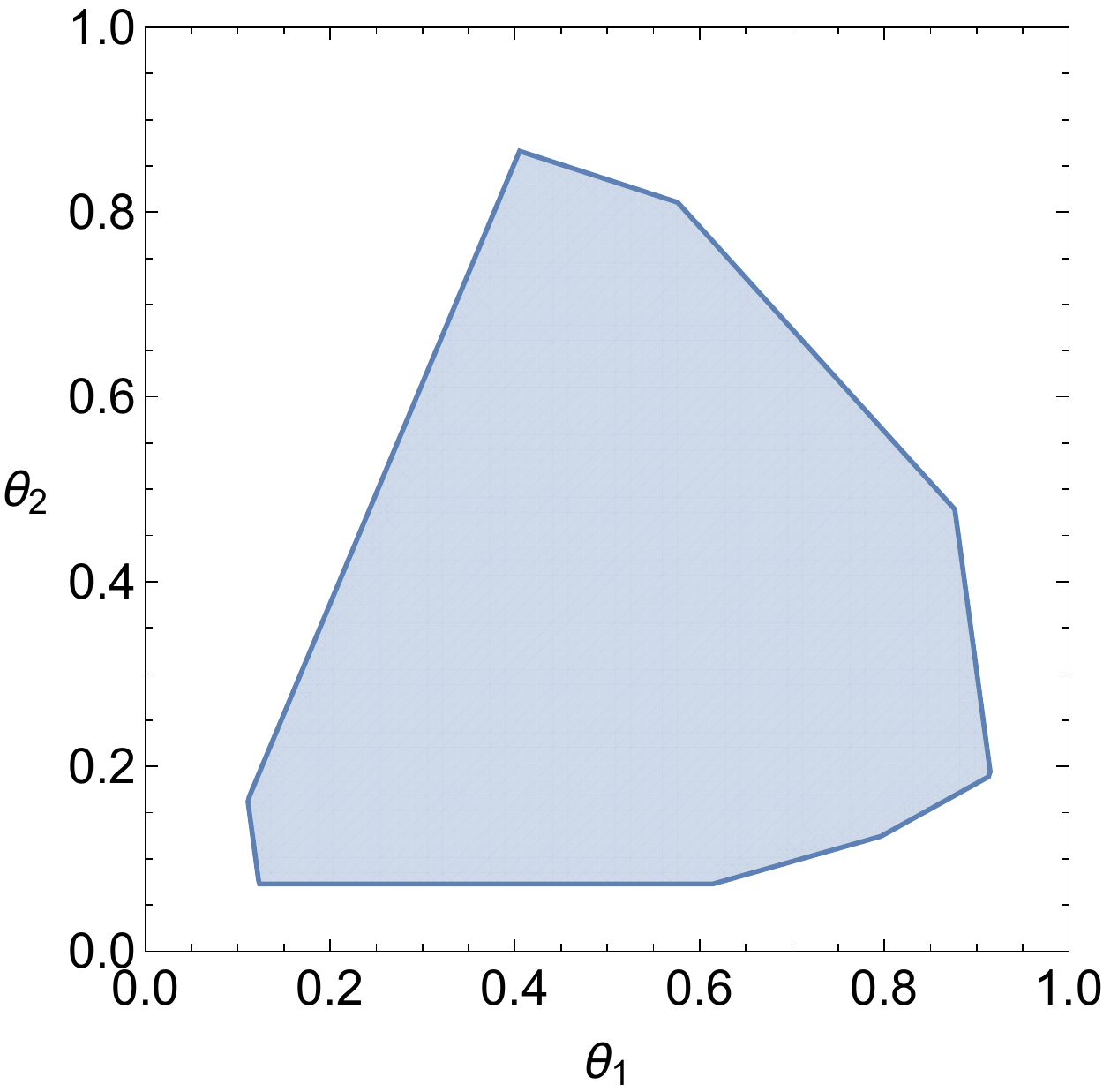}}
  \hfill
  \subfloat[Comparison with numerically computed period-$5$ orbits.\label{fig/ir5-comparison}]%
  {\includegraphics[width=0.48\linewidth]{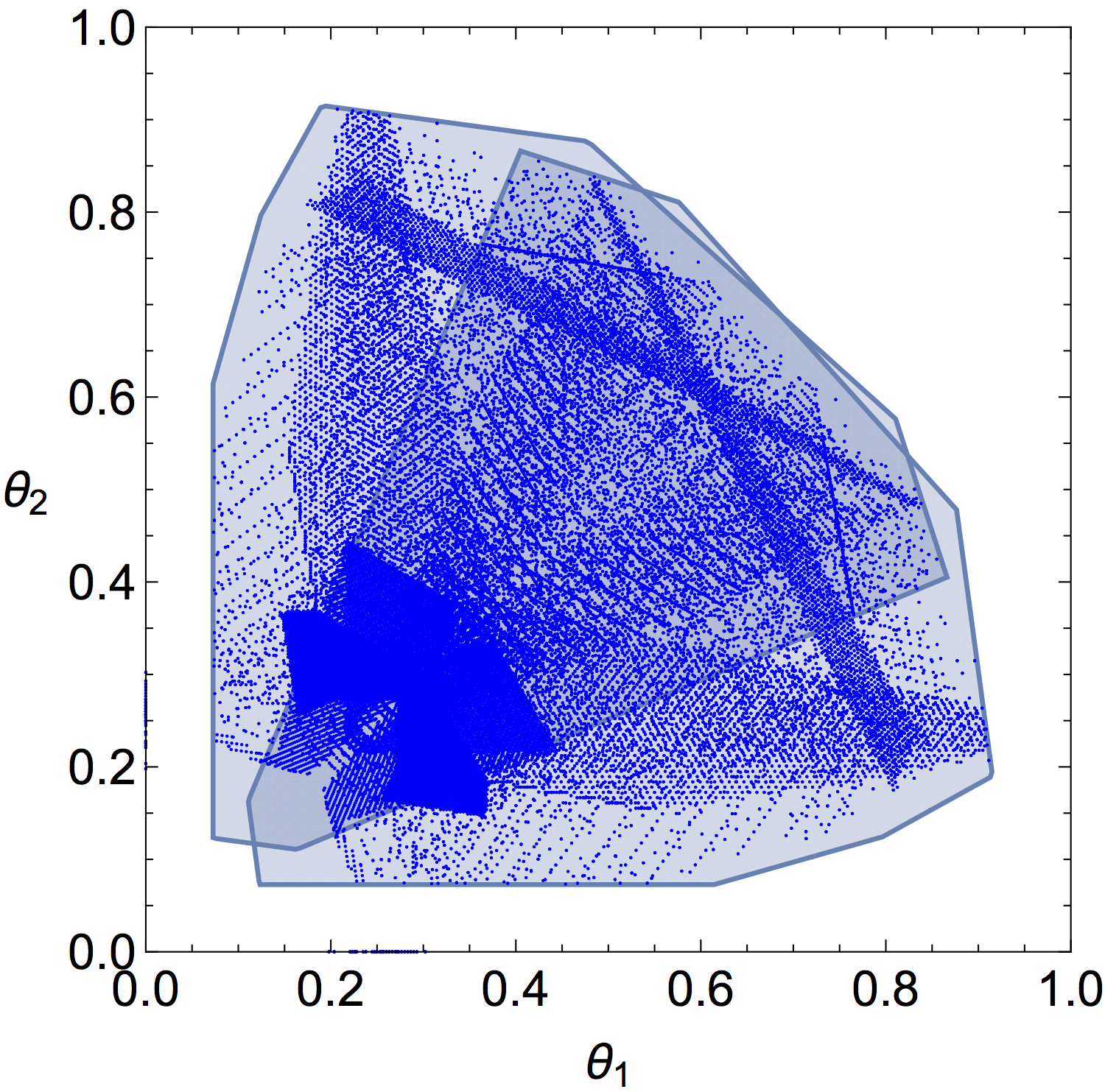}}
  \caption{The set $\mathcal A_{\varepsilon,\tau}$ for IR5 (the projection of $S(\Omega_{\varepsilon,\tau})$ onto the $(\theta_1,\theta_2)$-plane) for IR5 and comparison with the numerically computed period-$5$ orbits. In the comparison we have drawn $\mathcal A_{\varepsilon,\tau}$ and its reflection with respect to the diagonal corresponding to the interchange of oscillators $1$ and $2$.}
\end{figure}

For the isochronous region IR5, corresponding to periodic orbits with Poincar\'e period $T_P=5$ we consider the subset $\Omega_{\varepsilon,\tau}$ of the $(\sigma_1, \sigma_2^{(1)}, \sigma_2^{(2)}, \sigma_3)$-space defined by the relations
\begin{subequations}
  \begin{align}\label{eq/ineq-pulse-class-C}
    \begin{aligned}
      & 0 < \sigma_2^{(1)} < \sigma_1 < \sigma_3 <\sigma_2^{(2)} < \tau, \\
      & H_* \le F_k(\bm{\sigma}; \tau) \le 1, \ \text{$k=1,2,3,4,5$},
    \end{aligned}
  \end{align}
  where
  \begin{align}\label{eq/def-F-C}
    \begin{aligned}
      F_1(\bm{\sigma};\tau) &:= H(\sigma_2^{(1)})+\tau-\sigma_3, \\
      F_2(\bm{\sigma};\tau) &:= H(\tau-\sigma_2^{(2)})+\sigma_2^{(2)}-\sigma_1, \\
      F_3(\bm{\sigma};\tau) &:= H(\sigma_2^{(2)}-\sigma_3)+\sigma_3-\sigma_2^{(1)}, \\
      F_4(\bm{\sigma};\tau) &:= H(\sigma_3-\sigma_1)+\sigma_1, \\
      F_5(\bm{\sigma};\tau) &:= H(\sigma_1-\sigma_2^{(1)})+\sigma_2^{(1)}+\tau-\sigma_2^{(2)}. \\
    \end{aligned}
  \end{align}
\end{subequations}
Then the set of initial states comprising IR5 is $S(\Omega_{\varepsilon,\tau})$ where $S$ is given by
\begin{align*}
  S & : (\sigma_1, \sigma_2^{(1)},\sigma_2^{(2)}, \sigma_3) \\
    & \mapsto \left( \theta_1, \theta_2, \theta_3 ;\; \{\sigma_1\}, \{ \sigma_2^{(1)},\sigma_2^{(2)} \}, \{\sigma_3\} \right) \\
    & = \Big( H(\sigma_1-\sigma_2^{(1)})+\sigma_2^{(1)}, H(\sigma_2^{(1)}), 0 ;\; \\
    & \qquad
    \{\sigma_1\}, \{ \sigma_2^{(1)},\sigma_2^{(2)} \}, \{\sigma_3\} \Big).
\end{align*}
The projection of $S(\Omega_{\varepsilon,\tau})$ on the $(\theta_1,\theta_2)$-plane is shown in \autoref{fig/ir5-theta}.

Using similar arguments as in the analysis of IR4 we find that the point
\begin{align*}
  \bm{\sigma}_*
  = \left( \sigma_1, \sigma_2^{(1)}, \sigma_2^{(2)}, \sigma_3 \right)
  = \left( \frac{2\tau}{5}, \frac{\tau}{5}, \frac{4\tau}{5}, \frac{3\tau}{5} \right).
\end{align*}
gives a periodic orbit with Poincar\'e period $T_P=1$. Its phase evolution is shown in \autoref{fig/ir5-center-dynamics}. Moreover, we find that this occurs for
\begin{align} \label{eq/master-ineq5}
  H_* \le H\left( \frac{\tau}{5} \right) + \frac{2\tau}{5} \le 1,
\end{align}
thus giving the subset of the parameter space $(\varepsilon, \tau)$ for which IR5 exists, see \autoref{fig/ir5-exist}.

\section{Conclusions}
\label{sec/conclusions}

We have reported the existence of non-trivial isochronous dynamics in pulse coupled oscillator networks with delay. In particular, we have presented numerical evidence for the existence of such isochronous regions and we have proved their existence for a subset of the parameter space $(\varepsilon,\tau)$ with non-empty interior. Moreover, we have described in detail the dynamics and stability of orbits in one of the isochronours regions that we call IR4.

The appearance of isochronous regions in pulse coupled oscillator networks with delays demonstrates the capacity of such systems for generating non-trivial dynamics that one would not, in general, expect for smooth dynamical systems. Of particular interest here is that isochronous dynamics coexists with attracting isolated fixed points and periodic orbits. This may be of interest for applications using heteroclinic connections between saddle periodic orbits as representations of computational tasks \citep{Ashwin2004, Ashwin2005b, Schittler-Neves2012}.

Several questions regarding isochronous regions in pulse coupled oscillator networks with delay remain open. The main questions going forward is whether such dynamics exist for larger numbers of oscillators and whether such dynamics persists in networks with non-identical oscillators or different network structure.

\section*{Acknowledgements}
This work was completed when P.L., supported by the China Scholarship Council, worked as a visiting PhD student at the University of Groningen.  W.L. was supported by the NSFC (Grants no. 11322111 and no. 61273014). K.E. was supported by the NSFC (Grant no. 61502132) and the XJTLU Research Development Fund (no. 12-02-08).

\bibliography{library}

\end{document}